\RequirePackage{fix-cm}
\documentclass[smallextended]{svjour3}       
\smartqed  
\usepackage{graphicx}
\usepackage{xcolor}
\usepackage{amssymb}
\usepackage{amsmath,bm}
\usepackage{booktabs}
\usepackage{algorithm}
\usepackage[noend]{algpseudocode}
\usepackage{tikz}
\usepackage{mathtools}
\usetikzlibrary{arrows.meta,positioning,shapes.geometric,calc}

\usepackage{url}
\usepackage[bookmarksnumbered=true]{hyperref}
\hypersetup{
	colorlinks,
	linkcolor={red!50!black},
	citecolor={blue!50!black},
	urlcolor={blue!80!black}
}

\newcommand{\SG}{\mathcal{S}_G}
\newcommand{\A}{\mathcal{A}}
\newcommand{\Ll}{\mathcal{L}}
\newcommand{\E}{\mathcal{E}}
\newcommand{\R}{\mathcal{R}}
\newcommand{\C}{\mathcal{C}}
\newcommand{\M}{\mathcal{M}}
\newcommand{\Paths}{\mathcal{P}}
\newcommand{\bz}{\mathbf{z}}
\newcommand{\bx}{\mathbf{x}}

\newcommand{\bu}{\mathbf{u}}
\newcommand{\be}{\mathbf{e}}
\newcommand{\bzero}{\mathbf{0}}
\newcommand{\btheta}{\bm{\theta}}
\newcommand{\bomega}{\bm{\omega}}
\newcommand{\bpi}{\bm{\pi}}
\newcommand{\st}{\text{s.t.}}
\newcommand{\Aout}{A^{\rm out}}
\newcommand{\Ain}{A^{\rm in}}
\newcommand{\rt}{$r$--$t$}
\newcommand{\var}{\operatorname{var}}
\newcommand{\val}{\operatorname{val}}
\newcommand{\rep}{\operatorname{rep}}
\newcommand{\conv}{\operatorname{conv}}

\DeclareMathOperator*{\argmax}{arg\,max}

\begin{document}

\title{
Beyond Hand-Derived Inequalities: Decision Diagrams for Cut Generation in Binary Polynomial Optimization
}

\titlerunning{Decision Diagram Cuts for Binary Polynomial Optimization}        

\author{Martin Cooper         \and
        Margarita Castro 
}


\institute{M. Cooper \at
              Department of Industrial and Systems Engineering, Pontificia Universidad Católica de Chile \\
           \and
           M. Castro \at
            Department of Industrial and Systems Engineering, Pontificia Universidad Católica de Chile \\
            \email{margarita.castro@uc.cl}  
}

\date{Received: date / Accepted: date}

\maketitle

\begin{abstract}

We study cutting-plane generation for binary polynomial optimization (BPO), whose feasible region is the multilinear set of a hypergraph. Strong inequalities for this set---such as two-links, flowers, and odd $\beta$-cycles---are classically hand-derived for fixed support patterns. Instead, we propose a decision-diagram (DD) approach: for any chosen support, it separates a facet-defining cut in the local multilinear polytope and lifts it back to the original problem. We utilize a novel compact DD encoding based on a recursive formulation that represents only the vertex variables and implicitly encodes the hyperedge variables inside the state representation. From this DD encoding, we also obtain: (i) an extended formulation of the multilinear polytope, (ii) a width characterization via valid antichains that uncovers a new polynomially solvable class of hypergraphs, and (iii) a certificate for the facetness of the generated cuts. We explore three different support strategies for cut generation that reuse, expand, or partition local structures into section hypergraphs. Our empirical results show that our DD-based methodologies achieve a larger gap closure at the root node with fewer cuts than existing procedures and, in turn, markedly accelerate branch-and-bound procedures.

\end{abstract}

\section{Introduction}
\label{sec:introduction}

Binary polynomial optimization (BPO) is a combinatorial problem that seeks a binary point that optimizes a multivariate polynomial function. The problem arises in many fields such as engineering, computer science, statistics, and
physics~\cite{barahonaGrotschelJungerReinelt1988,borosHammer2002},
and it generalizes well-studied special cases such as unconstrained binary quadratic
optimization~\cite{kochenberger2014ubqp} and the maximum cut
problem~\cite{barahonaMahjoub1986}.  

Formally, the problem can be encoded by a hypergraph 
$G=(V,E)$, in which each node $v\in V$ represents a binary variable (i.e., $z_v \in \{0,1\}$) and each hyperedge $e\in E$ represents a monomial of the objective encoded by a binary variable (i.e., $z_e \in \{0,1\}$). In what follows, we use bold notation $\bz_v= (z_v)_{v\in V}$ and $\bz_e= (z_e)_{e\in E}$ to refer to the vectors associated with vertex and hyperedge variables, respectively. Then, BPO minimizes a linear objective over the multilinear set $\SG = \left\{(\bz_v,\bz_e)\in \{0,1\}^{V\cup E} : \;
z_e=\prod_{v\in e} z_v \ \forall e\in E\right\}$ of hypergraph $G$ \cite{del2017polyhedral}, that is,
\begin{equation}\label{model:bpo}
\min \left\{ \sum_{v\in V} c_v z_v + \sum_{e\in E} c_e z_e : \; (\bz_v,\bz_e) \in \SG  \right\},
\end{equation}
where \(c_v\) and $c_e$ are the costs associated with each vertex $v \in V$ and hyperedge $e \in E$, respectively.

This problem is commonly solved by linearizing $\SG$ and strengthening the resulting relaxation with valid inequalities. Several important inequality families are tied to such hypergraph structures---including
the two-link~\cite{cramaRodriguezHeck2017valid}, flower~\cite{delPiaKhajavirad2018acyclic},  running-intersection~\cite{delPiaKhajavirad2021running}, and odd \(\beta\)-cycle
inequalities~\cite{delPiaDiGregorio2021chvatal,delPiaWalter2024simple}---which strengthen the standard linearization by capturing interactions among overlapping multilinear terms. At the same time, each family is hand-derived for a particular support pattern. This leaves open the question of how to generate strong inequalities once the relevant local interactions no longer fit a small catalog of named templates.

To address this, decision diagrams (DDs) offer a flexible alternative for generating such general cuts. A DD is a layered directed acyclic graph whose root-to-terminal paths encode the feasible solutions of a discrete set, compactly capturing discrete structures that a linear programming (LP) relaxation does not~\cite{bergman2016decision,castroCireBeck2022survey,vanHoeve2024intro}. Building on early work that embedded DDs in a branch-and-cut framework~\cite{beckerBehleEisenbrandWimmer2005,behle2007}, Tjandraatmadja and van Hoeve~\cite{tjandraatmadjaVanHoeve2019target} introduce \emph{target cuts}: given a DD, they separate inequalities that are facet-defining for the convex hull of the points it encodes, so the strength of the cut is governed by how faithfully the DD models the relevant substructure. Other related DD-based schemes include an outer approximation procedure for non-linear problems~\cite{davarniaVanHoeve2021outer} and the combinatorial cut-and-lift for binary problems~\cite{castroCireBeck2022cutLift}, while specialized procedures have also been developed, for instance, to generate cuts for stochastic optimization~\cite{casassus2025decision,lozano2022binary,macneil2024leveraging}. 

This makes DDs an appealing engine for BPO: rather than deriving a valid inequality for each named support pattern, one can build a DD for a chosen support hypergraph and generate target cuts that expose its local interactions; thus, yielding a single mechanism that is both stronger and more versatile than any fixed template. Moreover, we can show that the resulting cuts are facet-defining for the support hypergraph and, under certain conditions, facets of the original problem.  

To generate such DD-based cuts, we present a novel compact DD encoding for BPO. Specifically, we present a recursive model over a given vertex order, in which vertex variables are explicitly represented as arcs in the diagram, while hyperedge variables are implicitly encoded within the problem's state representation. Then, the DD is a compact state-space representation of such a recursive model, and the target cut procedure is adapted to account for the implicit hyperedge variables.

Moreover, we use the DD representation beyond cut generation to: (i) create a novel extended formulation for BPO, (ii) uncover new hypergraph structures that can be solved in polynomial time, and (iii) build certificates for facet-defining constraints. Specifically, we adapt the DD-based network flow formulation to obtain an extended formulation that accounts for both vertex and hyperedge variables. We also analyze the DD size (i.e., the node count) by characterizing the number of states in the recursive model via valid antichains. This allows us to prove that several known structures (e.g., flowers and cycles) can be represented by a DD of bounded size. Furthermore, we define a new family of hypergraphs that lead to polynomial-size DDs, and, as such, can be solved in polynomial time using, for example, the DD-based extended formulation. Lastly, we develop a novel procedure to determine whether a DD-based cut is facet-defining by computing its dimensionality over the support of the DD, and we use existing theoretical results~\cite{del2017polyhedral} to prove that it is a facet in the original problem.

We instantiate this framework with three strategies for choosing the local support fed to the separator: a \emph{fixed-support} baseline that reuses the supports of known literature
inequalities, an \emph{expanded-support} rule that grows a seed support by overlap, and a \emph{partition-support} scheme that builds section hypergraphs from a vertex partition and is backed by zero-lifting guarantees. We evaluate these strategies on standard benchmarks from the literature (i.e., the LABS and VISION datasets), on which we observed a clear superiority of our procedures over existing methods. Specifically, expanded and, especially, partition
supports close substantially more of the root gap with far fewer cuts---up to $92.4\%$ of the
gap on LABS and the full gap on VISION---and these stronger relaxations carry over to
branch-and-bound (B\&B), where the partition setting gives the best practical performance.

The remainder of the paper is as follows. Section~\ref{sec:bpo_and_structures} states the BPO linear formulation and the most commonly used cutting planes from the literature. Section~\ref{sec:dds_for_bpo} introduces the compact DD encoding of BPO and its extended formulation. Section~\ref{sec:bdd_size_structures} characterizes the size of the compact DD through valid antichains, analyzes the DD size for different hypergraph structures, and uncovers a new hypergraph family that can be solved in polynomial time. Section~\ref{sec:dd_separation} describes the target-cut separation over our compact DD encoding, and Section~\ref{sec:support_expansion} introduces support expansion and partition strategies. Section~\ref{sec:dimension_facet_audit} develops the dimension and facet-audit procedure. Lastly, Section~\ref{sec:experimental_design} describes the computational experiments, and Section~\ref{sec:discussion} concludes our work.

\section{Binary polynomial optimization and existing cutting planes}
\label{sec:bpo_and_structures}

We now present the integer linear programming (ILP) formulation commonly used in the literature to solve the BPO problem \eqref{model:bpo} and review existing cutting-plane approaches for tightening it. Recall that instances of BPO can be encoded by a hypergraph $G=(V, E)$, where we assume  $|e|\ge 2$ for every $e\in E$, as linear terms are represented directly by the vertex variables. The standard linearization replaces the multilinear equalities encoded in $\SG$ with two sets of linear inequalities, resulting in the following model~\cite{fortet1960algebre,gloverWoolsey1974converting}:
\begin{subequations}\label{model:bpo-linear}
	\begin{align}
		\min & \sum_{v\in V} c_v z_v + \sum_{e\in E} c_e z_e \label{eq:bpo-objective} \\
		\st\  & z_e \le z_v & \forall e\in E,\ v\in e,\\
		& z_e \ge \sum_{v\in e} z_v - |e| + 1 &\forall e\in E, \\
		& z_v, z_e \in \{0,1\} &\forall v\in V,\ e\in E. \label{eq:bpo-binary-vars}
	\end{align}
\end{subequations}

Note that the optimal value of problem \eqref{model:bpo} can be obtained by optimizing over the multilinear polytope, denoted as $ MP_G=\conv(\SG)$, where $\conv(\cdot)$ is the convex hull operator of a set. Similarly, model \eqref{model:bpo-linear} optimizes over the same set, but by representing the multilinear set with linear inequalities. Relaxing \eqref{eq:bpo-binary-vars} with continuous variables yields the standard  LP relaxation, denoted by
$MP_G^{LP}$. This is the base relaxation used throughout this paper, which is later strengthened by cutting-plane algorithms.

\subsection{Valid linear inequalities to strengthen $MP_G^{LP}$}
\label{subsec:literature_inequalities}

It is well-known that $MP_G^{LP}$ is a weak relaxation of $MP_G$, and the gap between them is driven by interactions among overlapping hyperedges~\cite{cramaRodriguezHeck2017valid,del2017polyhedral}. Thus, several works have identified special hypergraph structures to build valid inequalities that separate fractional points from $MP_G$. We now briefly review three families of such techniques that provide natural points of comparison for our proposed DD-based procedure. Moreover, we use such structures as a basis to construct more complex support hypergraphs for our own DD-based cuts (see Section \ref{sec:support_expansion}). We refer the reader to the sources for all theoretical and algorithmic details. 

\paragraph{Two-link inequalities~\cite{cramaRodriguezHeck2017valid}.} A two-link
is a pair of hyperedges $e,f\in E$ with $e\neq f$ and $e\cap f\neq\emptyset$. The corresponding inequality couples the two overlapping monomials $e,f\in E$ as follows:
\[
\sum_{v\in f\setminus e} z_v + z_e - z_f \le |f\setminus e|.
\]

\paragraph{Flower inequalities~\cite{delPiaKhajavirad2018acyclic}.}
A flower consists of a center hyperedge $f\in E$ and a nonempty set
$T\subseteq E\setminus\{f\}$ of hyperedges denoted as petals. Each petal meets the center (i.e., $(e\cap f)\neq\emptyset \ \; \forall e\in T$) and the petal-center intersections are pairwise disjoint (i.e., $(e\cap f)\cap(e'\cap f)=\emptyset \ \; \forall e\ne e'\in T$). Then, the inequality is given by:
\[
\sum_{v \in f \setminus \cup_{e \in T} e} z_v + \sum_{e \in T} z_e - z_f
\le
\left|f \setminus \bigcup_{e \in T} e\right| + |T| - 1.
\]
Intuitively, a flower generalizes a two-link by relating the center to several petals at
once rather than one at a time. When the underlying support is $\gamma$-acyclic, the flower inequalities together with the standard linearization fully describe $MP_G$~\cite{delPiaKhajavirad2018acyclic}. We note that running-intersection 
inequalities~\cite{delPiaKhajavirad2021running} are a strict generalization of the flower inequalities, differing only in the more general overlap patterns of (kite-free) $\beta$-acyclic supports. Therefore, we consider only flower inequalities in our empirical comparison.

\paragraph{Odd $\beta$-cycle inequalities~\cite{delPiaDiGregorio2021chvatal}.}
Structurally, a $\beta$-cycle consists of distinct hyperedges
$(e_1,\ldots,e_k)$ and vertices $(v_1,\ldots,v_k)$, indexed cyclically, such
that $v_i\in e_i\cap e_{i+1}$ for every $i$, while nonconsecutive hyperedges in the
sequence do not intersect. Odd $\beta$-cycle inequalities are associated with odd such
cycles and require additional parity or sign data beyond this intersection pattern. We note that for our comparison, we use the regular odd $\beta$-cycle inequalities~\cite{delPiaDiGregorio2021chvatal}, not the simple odd $\beta$-cycle variant of Del Pia and Walter~\cite{delPiaWalter2024simple}, which is a weakened version of the former.

\section{Decision Diagrams for BPO}
\label{sec:dds_for_bpo}

This section introduces the DD encoding used for BPO cut generation. In the context of integer and combinatorial optimization, DDs are graphical structures that encode the solution set of an optimization problem as paths. In what follows, we present the standard DD terminology~\cite{bergman2016decision,castroCireBeck2022survey,vanHoeve2024intro}.

Consider a combinatorial problem with $n$ binary variables and feasible set $X_B\subseteq\{0,1\}^n$.
A binary DD $B=(N,A)$ is a layered directed acyclic graph that represents $X_B$ as follows. Its node set is partitioned into layers $N=(N_1,\ldots,N_{n+1})$, where $N_1=\{r\}$ and $N_{n+1}=\{t\}$ contain the root and terminal node, respectively. Each transition from layer $i$ to layer $i+1$ corresponds to choosing the value of the
$i$-th variable in the order given to the DD. Thus, an arc $a=(u,u')\in A$, with
$u\in N_i$ and $u'\in N_{i+1}$, is labeled by a value $\val(a)\in\{0,1\}$. Each node
has at most one outgoing arc with value $0$ and at most one with value $1$. Then, a root-terminal path $p=(a_1,\ldots,a_n)$ encodes the binary vector $\bx^p=(\val(a_1),\ldots,\val(a_n))\in\{0,1\}^n$. Let $\mathcal{P}$ denote the set of all root-terminal paths; then we say that the DD is \emph{exact} if $X_B=\{\bx^p:p\in\mathcal{P}\}$, that is, each root-terminal path has a one-to-one correspondence with the solutions in $X_B$. 


\subsection{A compact DD encoding for BPO} \label{subsec:compact_DD_bpo}

We now introduce our compact DD representation of BPO. This DD encoding is based on a dynamic programming (DP) model for BPO that considers only vertex variables as state variables, with hyperedge variables represented implicitly in the states. The DP model leverages that $z_e$ for $e \in E$ is completely defined by the assignment of its corresponding nodes since $z_e = \prod_{v \in e} z_v$.

We start by introducing some useful notation for the DP model. Consider a hypergraph $G=(V,E)$ that represents a BPO instance and a given vertex order $\sigma=(v_1,\ldots,v_n)$, where $n = |V|$. For $i\in \{0,\ldots,n\}$, let $U_0=\emptyset$ and $U_i=\{v_1,\ldots,v_i\}$, be the set of the first $i$ vertex decisions based on $\sigma$.  Then, we define the set of \textit{active hyperedges} with respect to $U_i$ as
$\A_i=\{e\in E : e\cap U_i\neq\emptyset,\ e\nsubseteq U_i\}$, that is, the set of hyperedges that have some vertex in $U_i$, but not all of them. We say that a hyperedge $e\in E$ is \emph{compatible} with the current partial assignment over the vertices in $U_i$ if every assigned vertex of $e$ has value one. 

We define our DP model with $n$ stages (i.e., one per vertex) such that the state variable at stage $i \in \{1,\ldots,n\}$ corresponds to the vertex variable $z_{v_i}\in\{0,1\}$ following the $\sigma$ order. State $s_i$ at stage $i \in \{1,\ldots,n\}$ is a subset of active hyperedges $\A_i$ that are compatible with respect to the current partial assignment. Given $s_0=\emptyset$ as the initial state, the transition function $\phi_i(s_{i-1},z_{v_i})$ at stage $i \in \{1,\ldots,n\}$ is given by
\[
\phi_i(s_{i-1},z_{v_i})=
\left\{
e\in\A_i :
\bigl(e\notin\A_{i-1}\ \text{or}\ e\in s_{i-1}\bigr)
\text{ and }
\bigl(v_i\notin e\ \text{or}\ z_{v_i}=1\bigr)
\right\}.
\]
The first condition considers the previously active hyperedges and potentially newly active hyperedges at stage $i$, while the second condition updates compatibility after assigning $v_i$. Intuitively, this encoding keeps track of the $\bz_e$ variables that can take the value one in each state (i.e., a compatible hyperedge) and ignores them if we have already assigned all vertices associated with a specific hyperedge. Formally, let
$\Ll_i=\{e\in E : v_i\in e,\ e\subseteq U_i\}$ be the hyperedges whose last vertex in  $\sigma$ is $v_i$ for every $i \in \{1,\ldots,n\}$. For every arc that assigns $z_{v_i}$, a hyperedge $e \in \Ll_i$ has its value assigned: $z_e = 1$ if it is compatible before assigning $v_i$ and $z_{v_i}=1$; and $z_e = 0$ otherwise.

Given this DP model, the DD is built in a top-down fashion from the root (i.e., initial state $s_0$) by applying the transition function at each layer
and merging nodes with equal states. Exactness with respect to $\mathcal{S}_G$
follows because the state keeps the needed information to determine which unfinished hyperedges can still attain value $1$ under future assignments, and therefore to recover every vector $(\bz_v,\bz_e)$. See Example \ref{exa:bpo-dd} for an illustrative example.

\begin{example}\label{exa:bpo-dd}
Consider a BPO instance with vertices $\{v_1, v_2, v_3\}$ and hyperedges $e_1=\{v_1, v_2\}$ and $ e_2=\{v_2, v_3\}$ and its associated compact DD as depicted in Figure \ref{fig:example-dd}. Nodes depict the states, solid black arcs are $1$-arcs (i.e., $z_{v_i}=1$) and dashed gray arcs are $0$-arcs (i.e., $z_{v_i}=0$). Given order $\sigma=(v_1,v_2,v_3)$, we have $U_1 = \{v_1\}$, $U_2 = \{v_1, v_2\}$, $U_3 = \{v_1, v_2, v_3\}$ and active hyperedge sets $\A_1 = \{e_1\}$, $\A_2 = \{e_2\}$, and $\A_3 = \emptyset$. Note that the states of the DD indeed encode the set of hyperedges $e\in E$ that could have $z_e = 1$ but have not yet been assigned. For instance, after assigning $z_{v_1}$, we have a state where either $e_1$ still has options of being selected (i.e., when $z_{v_1} =1$) or not (i.e.,  when $z_{v_1} = 0$).
\end{example}

\begin{figure}
	\centering
\begin{tikzpicture}[
    vertex/.style={circle, draw, fill=white, font=\scriptsize,
                   minimum size=0.5cm, inner sep=1pt},
    hedge/.style={draw, line width=0.7pt, fill=gray!25, fill opacity=0.5}
  ]

  \draw[hedge] (0,-0.7) ellipse [x radius=0.55, y radius=1.15];
  \draw[hedge] (0,-2.1) ellipse [x radius=0.55, y radius=1.15];

  \node[font=\scriptsize] at (-0.95,-0.45) {$e_1$};
  \node[font=\scriptsize] at (-0.95,-2.35) {$e_2$};

  \node[vertex] (v1) at (0,  0  ) {$v_1$};
  \node[vertex] (v2) at (0, -1.4) {$v_2$};
  \node[vertex] (v3) at (0, -2.8) {$v_3$};

\end{tikzpicture}
    \hspace{5em}
\begin{tikzpicture}[
    >=Stealth, thick,
    state node/.style={ellipse, draw, fill=gray!12, font=\scriptsize,
                       minimum width=1.05cm, minimum height=0.55cm, inner sep=1pt},
    zero arc/.style={draw, dashed, gray, ->, line width=0.5pt},
    one arc/.style={draw, ->, line width=0.6pt},
    layer label/.style={font=\scriptsize\itshape, gray}
  ]

  \node[state node] (r)  at ( 0,    0  ) {$\emptyset$};
  \node[state node] (u1) at (-1, -1) {$\emptyset$};
  \node[state node] (u2) at ( 1, -1) {$\{e_1\}$};
  \node[state node] (u3) at (-1, -2) {$\emptyset$};
  \node[state node] (u4) at ( 1, -2) {$\{e_2\}$};
  \node[state node] (t)  at ( 0,   -3) {$\emptyset$};

  \node[layer label] at (-2,-0.5) {$v_1:$};
  \node[layer label] at (-2,-1.5) {$v_2:$};
  \node[layer label] at (-2,-2.5) {$v_3:$};

  \path
    (r)  edge[zero arc] (u1)
         edge[one arc]  (u2)
    (u1) edge[zero arc] (u3)
         edge[one arc]  (u4)
    (u2) edge[zero arc] (u3)
         edge[one arc]  (u4)
    (u3) edge[zero arc, bend right=18] (t)
         edge[one arc,  bend left=18]  (t)
    (u4) edge[zero arc, bend right=18] (t)
         edge[one arc,  bend left=18]  (t);

  \node (l0a) at (2.3,-0.15) {};
  \node[anchor=west, font=\scriptsize] (l0b) at (2.7,-0.15) {$z_{v_i}=0$};
  \node (l1a) at (2.3,-0.65) {};
  \node[anchor=west, font=\scriptsize] (l1b) at (2.7,-0.65) {$z_{v_i}=1$};
  \path (l0a) edge[zero arc] (l0b.west)
        (l1a) edge[one arc]  (l1b.west);

\end{tikzpicture}
	\caption{BPO instance as a hypergraph (left) and its associated DD encoding (right).}
	\label{fig:example-dd}
\end{figure}
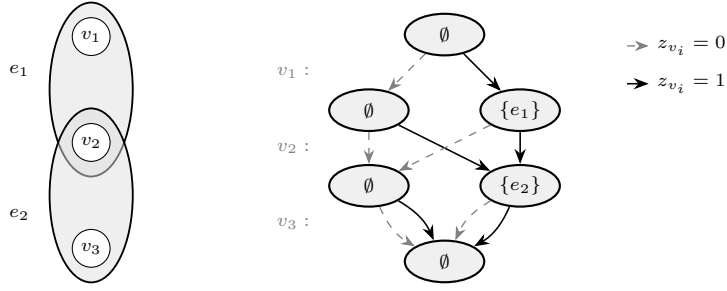

\subsection{An extended formulation induced by the compact DD}
\label{subsec:compact_DD_extended_formulation}

A simple way to use an exact DD inside an LP formulation is through its network-flow representation \cite{behle2007}. This formulation assumes that every variable of the problem is explicitly represented in the DD (i.e., it corresponds to one of the decisions in each layer).  In what follows, we adapt such a model to our compact DD encoding, in which hyperedge variables are implicitly represented. 

Let $B=(N,A)$ be the compact DD associated with $G=(V,E)$ and  assume vertex order $\sigma=(v_1,\ldots,v_n)$. For an arc $a\in A$, let $s(a)$ and $t(a)$ denote its source and target nodes, respectively, and consider sets $\Aout(u)=\{a\in A:s(a)=u\}$ and $\Ain(u)=\{a\in A:t(a)=u\}$ to represent outgoing and incoming arcs of a node $u \in N$. Lastly, if $a \in A$ leaves layer $N_i$, then $\var(a)=v_i$ is the vertex associated with arc $a$ and $\val(a)\in\{0,1\}$ is its assigned value.

We now define the arcs that assign each variable $(\bz_v, \bz_e)$ to the value one or zero. For a vertex $v\in V$, let $\Gamma_v^1=\{a\in A:\var(a)=v,\ \val(a)=1\}$ and $\Gamma_v^0=\{a\in A:\var(a)=v,\ \val(a)=0\}$ be the set of arcs associated with the $z_v = 1$ and $z_v = 0$ decisions, respectively. For hyperedge variables, recall that $\Ll_i$ is the set of hyperedges whose value can be deduced when $v_i$ is assigned, for each $i \in \{1,\ldots,n\}$. For each arc $a \in \Aout(u)$ leaving a node $u \in N_{i}$ with state $s_{i-1}$, we define
$
\E_1(a)=
\{e\in\Ll_i:e\in s_{i-1},\ \val(a)=1\}$,
and 
$\E_0(a)=\Ll_i\setminus\E_1(a)
$
as the sets of hyperedges assigned value 1 and 0 on arc $a$, respectively. Lastly, for each $e\in E$, $\Gamma_e^1=\{a\in A:e\in\E_1(a)\}$ and $\Gamma_e^0=\{a\in A:e\in\E_0(a)\}$ are the set of arcs that define the $z_e$ variable to 1 and 0, respectively. 

To create the extended formulation for BPO based on the DD $B=(N,A)$, consider $y_a\geq0$ a continuous flow variable associated with arc $a\in A$. Then, the network flow model associated with $B$ is:
\begin{subequations}\label{model:bpo-extended-dd}
\begin{align}
    \min\ & \sum_{v\in V} c_v z_v + \sum_{e\in E} c_e z_e \label{eq:flow-objective} \\
   \st\ & \sum_{a\in \Aout(u)} y_a - \sum_{a\in \Ain(u)} y_a = 0 & \forall u\in N\setminus\{r,t\}, \label{eq:flow-all}\\
   & \sum_{a\in \Aout(r)} y_a = \sum_{a\in \Ain(t)} y_a = 1, \label{eq:flow-root-terminal}\\
   & z_j = \sum_{a\in\Gamma_j^1} y_a & \forall j \in V\cup E, \label{eq:flow-link}\\
   & y_a \geq 0 & \forall a \in A. \label{eq:flow-vars}
\end{align}
\end{subequations}

Constraints \eqref{eq:flow-all} and \eqref{eq:flow-root-terminal} are the standard network-flow constraints over a DD, whose feasible flows are exactly the convex combinations of root-terminal paths; thus, they encode $\conv(X_B)$. Constraints \eqref{eq:flow-link} link the flow to the BPO variables: $z_v$ accumulates the flow on the arcs that set $z_v=1$, and $z_e$ accumulates the flow on the arcs where the compact DD implicitly finalizes $z_e=1$. Since every root-terminal path crosses exactly one arc of $\Gamma_j^1\cup\Gamma_j^0$ for each $j\in V\cup E$, these identities are well defined, and the objective \eqref{eq:flow-objective} coincides with the BPO objective \eqref{eq:bpo-objective}. The validity of the model and the fact that the projection of \eqref{eq:flow-all}--\eqref{eq:flow-vars} onto the $(\bz_v,\bz_e)$ variables equals $MP_G$ follow directly from the compact DD encoding and Theorem 4.1 of~\cite{behle2007}.

\section{DD size and special structures}
\label{sec:bdd_size_structures}

This section analyzes the DD size and uncovers special structures for which its size remains bounded. In what follows, we use the width of a DD $B=(N,A)$ to measure its size. Specifically, the width of layer $N_i$ for any $i \in \{1,\ldots,n\}$ is the number of nodes in such layer (i.e., $|N_i|$), and the width of $B$ is the maximum width over all layers, that is,  $\max_{i \in \{1,\ldots,n\}} |N_i|$. 

While the compact DD encoding presented in Section~\ref{sec:dds_for_bpo} gives us an extended formulation for BPO, it is exponential in size for general BPO instances, and thus, impractical. Indeed, since the DD nodes correspond to states that store subsets of hyperedges and we have at most $n = |V|$ binary variables, their width is, in the worst case, $O(2^{n-1})$. Nonetheless, we can identify structural conditions under which the DD size remains small, and, as shown in Section \ref{sec:dd_separation}, construct strong valid inequalities for these structures. In what follows, we use the set of compatible active hyperedges in a layer to characterize the size of a DD, which translates directly to DD width. We then use these results to identify hypergraph structures that can encode the BPO problem using bounded DD width and relate them to results in the literature.

\subsection{Valid antichains and compact DD width}
\label{subsec:states_valid_antichains}

We now analyze the width of the compact DD relying on the fact that active hyperedges do not behave independently. Indeed, if one hyperedge contains another, then the compatibility of the larger one forces compatibility of the other. Thus, the number of reachable states in a layer can be less than $O(2^{|E|})$ for many hypergraphs.

In what follows, we use standard terminology for posets, chains, antichains, and ideals used in discrete mathematics~\cite{stanley2012enumerative}. A \emph{poset} is a set equipped with a reflexive, antisymmetric, and transitive order. In our case, we consider set inclusions and illustrate these concepts with the poset $(2^{\{1,2,3\}}, \subseteq)$ as a running example.  A \emph{chain} is a subset of a poset in which each pair of elements is comparable, for example, $\{\{1\},\{1,2\}, \{1,2,3\}\}$ is a chain since $\{1\}\subseteq \{1,2\}\subseteq \{1,2,3\} $. An \emph{antichain} is a set of pairwise incomparable elements (e.g., $\{\{1,2\},\{2,3\}\}$). An \emph{ideal} is a downward-closed subset of a poset: whenever it contains an element, it also contains all smaller elements. Thus, if an ideal contains $\{1,2,3\}$, it must also contain every subset in the poset contained in $\{1,2,3\}$.

Now consider a DD with a vertex order $\sigma=(v_1,\ldots,v_n)$ and a layer $N_i$ for some $i \in \{1,\ldots,n\}$. For each active hyperedge $e\in\A_i$, we define its \emph{reduced active hyperedge} by $\rho_i(e)=e\cap U_i$, that is, ignoring all vertices except the ones in $U_i$. Let $\R_i=\{\rho_i(e):e\in\A_i\} $ be the set of distinct reduced active hyperedges at layer $N_i$, partially ordered by inclusion. If $R,R'\in\R_i$ are two reduced hyperedges such that  $R\subseteq R'$, then compatibility of $R'$ implies compatibility of $R$: if all assigned vertices in $R'$ have value one, then the same holds for $R$. Hence, the compatible reduced hyperedges in any reachable state form an ideal of the poset $(\R_i,\subseteq)$, and that ideal is determined by its maximal elements, which form an antichain. Note that not every antichain encodes a reachable state because the union of the antichain may force additional reduced hyperedges to be compatible. 

Given an antichain $\C$ of $(\R_i, \subseteq)$, we define
$ X(\C)=\bigcup_{R\in\C}R $ as the union of all reduced hyperedges in $\C$.
We call $\C$ a \emph{valid antichain} at layer $N_i$ if
\[
\forall R\in\R_i,\ R\subseteq X(\C)
\quad\Longrightarrow\quad
\exists M\in\C\text{ such that }R\subseteq M.
\]
In other words, an antichain is valid if every reduced hyperedge forced by setting the vertices in its union to one is already covered by one of its maximal elements. For example, if $ \{1,2\},\{2,3\},\{1,2,3\}\in\R_i$,
then $\{\{1,2\},\{2,3\}\}$ is an antichain but not a valid one, because setting elements $1,2,3$ to one also makes $\{1,2,3\}$ compatible. The valid antichain
representing that assignment is instead $\{\{1,2,3\}\}$.  Lemma \ref{lem:valid_antichains} uses these definitions and observations to characterize the reachable states at layer $N_i$ by valid antichains.

\begin{lemma}
\label{lem:valid_antichains}
Consider a vertex order $\sigma$ and a layer $N_i$ of the compact DD $B=(N,A)$. The reachable states after assigning $U_i$ variables are in bijection with the valid antichains of $(\R_i,\subseteq)$.
\end{lemma}

\begin{proof}
We first show that each reachable state in $N_i$ corresponds to a valid antichain in $(\R_i,\subseteq)$. Let $s_i$ be a reachable state and define $I(s_i)=\{\rho_i(e):e\in s_i\}\subseteq\R_i$ as the set of reduced active hyperedges that are compatible in a reachable state $s_i$. Set $I(s_i)$ is downward closed because if $R'\in I(s_i)$ and $R\in\R_i$ satisfies $R\subseteq R'$, then all vertices of $R$ also have value one under the current assignment, so $R$ is compatible as well. Therefore $R\in I(s_i)$, and $I(s_i)$ is an ideal.

Since every ideal is determined by its maximal elements, let us define
$
\operatorname{Max}(I(s_i)) = \{R\in I(s_i): \nexists R'\in I(s_i)\text{ with }R\subsetneq R'\}
$
as the set of maximal elements, which in turn form an antichain. Note that the elements in $\operatorname{Max}(I(s_i))$ are also valid: if
$R\in\R_i$ satisfies $R\subseteq X(\operatorname{Max}(I(s_i)))$, then all
vertices of $R$ have value one, so $R$ is compatible and hence belongs to $I(s_i)$. Because $I(s_i)$ is an ideal, $R$ is contained in some maximal element of $I(s_i)$.

Now, we show the other direction by assuming that $\C$ is a valid antichain. Consider the assignment on $U_i$ that sets the vertices in $X(\C)$ to one and the remaining vertices of $U_i$ to zero. Under this assignment, a reduced active hyperedge $R\in\R_i$ is
compatible if and only if $R\subseteq X(\C)$. By validity, every such $R$ is
contained in some member of $\C$. Hence, the induced state is
$
s_i(\C)=
\{e\in\A_i :
\rho_i(e)\subseteq M \text{ for some } M\in\C\}.
$
The maximal elements of the corresponding ideal are exactly the members of
$\C$. This gives the claimed bijection. \hfill$\square$
\end{proof}

Counting antichains in a general poset is \#P-complete~\cite{provanBall1983complexity},
so valid antichains are not themselves a practical algorithmic primitive. They are,
however, a useful structural guide: good orders are those that make reduced active
hyperedges comparable, or even identical, as often as possible.

We now use this notation to characterize a new family of hypergraphs that lead to polynomial-size DDs (see Theorem \ref{thm:vaw_polynomial_bdd}), and, in turn, BPO instances that can be solved in polynomial time with, for instance, the extended formulation \eqref{model:bpo-extended-dd}. To do so, let $\operatorname{VA}_i(\sigma)$ denote the set of valid antichains of
$(\R_i,\subseteq)$ at layer $N_i$ under order $\sigma$. By Lemma~\ref{lem:valid_antichains}, the compact-DD width at layer $N_{i}$ is $ |\operatorname{VA}_{i-1}(\sigma)|$. Therefore, the maximum width induced by $\sigma$ is $W(\sigma)=\max_i |\operatorname{VA}_i(\sigma)|$, and the valid-antichain width of $G$ is $\operatorname{VAW}(G)=\min_{\sigma} W(\sigma)$.

\begin{theorem}
\label{thm:vaw_polynomial_bdd}
Let $\mathcal{H}$ be a class of hypergraphs such that, for every
$G=(V,E)\in\mathcal{H}$, there exists a vertex order $\sigma$ satisfying
$W(\sigma) \le q(|V|,|E|) $
for some polynomial $q$. Then the compact DD encoding for $G$ under $\sigma$ has polynomial size. Moreover, if $q(|V|,|E|)$ is a constant, then the compact DD has $O(|V|)$ nodes and arcs.
\end{theorem}

\begin{proof}
By Lemma~\ref{lem:valid_antichains}, each layer $N_i$ contains exactly
$|\operatorname{VA}_{i-1}(\sigma)|$ reachable states, for each $i\in \{1,\ldots,n\}$ ($n = |V|$). Therefore, the number of non-terminal nodes is
\[
\sum_{i=1}^{n} |\operatorname{VA}_{i-1}(\sigma)|
\le |V|\,\max_{i\in \{0,\ldots,n-1\}} |\operatorname{VA}_i(\sigma)|  = |V|\,W(\sigma) \leq |V|\cdot  q(|V|,|E|),
\]
making the node count polynomial. Note that if $q(|V|,|E|) = q \in \mathbb{R}^+$ is a constant, the node and arc count is $O(|V|)$ because each node has at most two outgoing arcs. \hfill$\square$
\end{proof}

\subsection{Small exact DDs for literature structures}
\label{subsec:small_bdds_literature}

We now show that there exist well-known hypergraph structures with a bounded compact DD encoding. We use the previously introduced valid antichain characterization to find appropriate vertex orders that satisfy the conditions in Theorem \ref{thm:vaw_polynomial_bdd}. Specifically, we show that flowers and cycles have this property, as stated in Propositions \ref{prop:flower} and \ref{prop:cycle}, respectively. We focus on these two structures since flower~\cite{delPiaKhajavirad2018acyclic} and  $\beta$-cycle~\cite{delPiaDiGregorio2021chvatal,delPiaWalter2024simple} inequalities are based on them and are, thus, also attractive to our DD-based cutting plane approach. In what follows, a cycle hypergraph is one whose hyperedges can be ordered cyclically so that each hyperedge intersects only its two cyclic neighbors.



\begin{proposition}\label{prop:flower}
Let $G=(V,E)$ be a flower with center hyperedge $e_0$ and petals $e_1,\ldots,e_m$. Assume the petals are pairwise disjoint outside the center and that their intersections with the center are pairwise disjoint. Then, $\operatorname{VAW}(G)\le 4$.
\end{proposition}

\begin{proof}
We first define a vertex order $\sigma$ in which the vertices in the petals are considered first, followed by the vertices at the center. Formally, consider an arbitrary petal order and, for each petal $e_j$, the order first uses the vertices
in $e_j\cap e_0$, in any order, and then the vertices in $e_j\setminus e_0$, again in any order. After all petal vertices have been ordered, order any remaining vertices of
$e_0\setminus \bigcup_{j=1}^m e_j$ arbitrarily. Because the petals are pairwise disjoint outside the center and their intersections with the center are pairwise disjoint, this order ensures that at any layer there is at most one active petal together with the center hyperedge $e_0$.

Now consider a layer in which the active petal is $e_j$. While the order is processing the vertices of $e_j\cap e_0$, every assigned vertex of $e_j$ is also an assigned vertex of the
center, so $\rho_i(e_j)\subseteq \rho_i(e_0)$. Therefore, the set of valid antichains is
$\{
\emptyset,\
\{\rho_i(e_j)\},\
\{\rho_i(e_0)\}\}
$, so there are at most three reachable states.

Then, consider a layer where we assign vertices of $e_j\setminus e_0$. The reduced petal $\rho_i(e_j)$ contains assigned vertices outside the center. In general,
$\rho_i(e_j)$ and $\rho_i(e_0)$ are then incomparable. Since no other petal is active, the set of valid antichains is
$\{\emptyset,\
\{\rho_i(e_j)\},\
\{\rho_i(e_0)\},\
\{\rho_i(e_j),\rho_i(e_0)\} \}$. Thus, there are at most four states while a petal is active. Finally, after all petals
have closed, only the center hyperedge can remain active, giving at most two reachable states. Therefore, every layer has at most four valid antichains. By Lemma~\ref{lem:valid_antichains}, the compact DD constructed from this order has width at most $4$, hence $\operatorname{VAW}(G)\le 4$. \hfill$\square$
\end{proof}

\begin{proposition}\label{prop:cycle}
Let $G=(V,E)$ be a cycle. Then, $\operatorname{VAW}(G)\le 6$.
\end{proposition}

\begin{proof}
First, we construct the vertex order in accordance with the cyclic nature of the hyperedges.  Assume that $G=(V,E)$ has $|E|= L+1$ hyperedges, choose an arbitrary hyperedge $e_0$ of the cycle, and let $e_L$ be one of its two cyclic neighbors. Let $e_1,e_2,\ldots,e_{L-1}$ be the sequence of remaining hyperedges obtained by walking around the cycle in the other direction, so that $e_{L-1}$ is adjacent to $e_L$. Then, the vertex order $\sigma$ is constructed as follows. First, consider all vertices of $e_0$ and then, for
$h=1,\ldots,L-1$, assign all not-yet-assigned vertices of $e_h$. Finally, consider the remaining vertices of $e_L$. Ties inside each hyperedge can be broken arbitrarily.

We now characterize the reachable states in each layer using valid antichains. When the vertices of $e_0$ are being assigned, the only possible active hyperedges are $e_0$ and its two neighbors, $e_L$ and $e_1$. Moreover, we have the active hyperedge relations $\rho_i(e_L)\subseteq \rho_i(e_0)$, and $\rho_i(e_1)\subseteq \rho_i(e_0)$, because every assigned vertex of $e_L$ or $e_1$ belongs to $e_0$ at this stage.  Then, the set of  possible valid antichains is $\{ 
\emptyset,\ \{\rho_i(e_L)\},\  \{\rho_i(e_1)\},\ \{\rho_i(e_L),\rho_i(e_1)\},\ \{\rho_i(e_0)\} \}$. Thus, by Lemma \ref{lem:valid_antichains}, the maximum width in these first layers is at most $5$.

Now consider a layer in which the order assigns the remaining vertices of $e_h$, with $1\le h<L-1$. At this point, only the unresolved neighbor $e_L$,
the current hyperedge $e_h$, and the next hyperedge $e_{h+1}$ may be active. Thus, the only containment relation that can arise is $\rho_i(e_{h+1})\subseteq \rho_i(e_h)$ because the only assigned vertices of $e_{h+1}$ are those already assigned inside $e_h$. Therefore, the set of possible valid antichains is $\{\emptyset,\ \{\rho_i(e_L)\},\ \{\rho_i(e_{h+1})\},\ 
\{\rho_i(e_L),\rho_i(e_{h+1})\},$ $\{\rho_i(e_h)\},\ \{\rho_i(e_L),\rho_i(e_h)\} \}$. Thus, by Lemma \ref{lem:valid_antichains}, there are at most six reachable states in such a layer.

Finally, if the current vertex corresponds to a hyperedge adjacent to $e_L$ (i.e., $e_{L-1}$), at most two reduced active hyperedges remain relevant, namely those of $e_{L-1}$
and $e_L$. If they are incomparable, the set of possible valid antichains are $\{\emptyset,\ \{\rho_i(e_L)\},\ \{\rho_i(e_{L-1})\},\ \{\rho_i(e_L),\rho_i(e_{L-1})\} \} $, and if one contains the other there are fewer. Hence, in these final layers there are at most four reachable states.

Taken together, the three cases imply that every layer has at most six valid antichains. By Lemma~\ref{lem:valid_antichains}, the compact DD constructed from this order has width at most $6$, hence $\operatorname{VAW}(G)\le 6$.
\hfill$\square$
\end{proof}

Example \ref{exa:bounded-compact-dds} illustrates the vertex order $\sigma$ defined for each hypergraph structure and the corresponding compact DD of bounded width. Also, we note that while  Proposition \ref{prop:cycle} proves that cycle hypergraphs allow a compact DD encoding of bounded size, it remains open whether $\beta$-acyclic hypergraph classes admit polynomial-size compact DDs under suitable orders. This question is distinct from the known knowledge-compilation result that bounded incidence treewidth yields a polynomial-size d-DNNF representation and polynomial-size extended formulations for BPO \cite{capelliDelPiaDiGregorio2024kc}. As shown in the following subsection, incidence-treewidth arguments do not directly carry over to our compact DD construction.

\begin{example}\label{exa:bounded-compact-dds}
Figure~\ref{fig:dd-structures} depicts the vertex orders of Propositions~\ref{prop:flower} and~\ref{prop:cycle} and the resulting bounded-width compact DDs. Figure~\ref{fig:dd-structures}(a) illustrates the two hypergraphs: on top a flower with center $e_0=\{v_1,v_3,v_5\}$ and petals $e_1=\{v_1,v_2\}$ and $e_2=\{v_3,v_4\}$, and on bottom a cycle with hyperedges $e_0=\{v_1,v_2,v_3\}$, $e_1=\{v_2,v_4\}$, $e_2=\{v_4,v_5\}$, and $e_3=\{v_3,v_5\}$. Note that vertices are numbered according to the vertex order of each DD. In both cases, the  compact DD width is bounded and slightly smaller than the worst-case scenario proved in the corresponding propositions. 
\end{example}

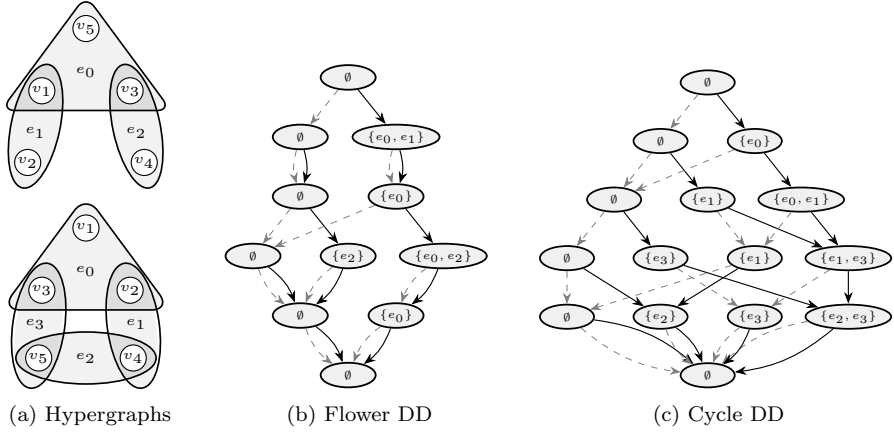
\begin{figure}[t]
    \centering
    \begin{minipage}[b]{0.20\linewidth}\centering
        \resizebox{\linewidth}{!}{\begin{tikzpicture}[
    scale=1.0,
    every node/.style={transform shape},
    vertex/.style={
        circle, draw, fill=white, font=\tiny,
        minimum size=0.33cm, inner sep=0.4pt
    },
    hfill/.style={
        draw=none, fill=gray!22, fill opacity=0.45
    },
    hint/.style={
        draw=none, fill=gray!50, fill opacity=0.28
    },
    hborder/.style={
        draw, line width=0.6pt, fill=none
    },
    elabel/.style={font=\tiny}
]


  \fill[hfill, rounded corners=5pt]
      ( 0.00,  0.22)
   -- (-1.08, -1.24)
   -- ( 1.08, -1.24)
   -- cycle;

  \fill[hfill, rotate around={-12:(-0.64,-1.44)}]
      (-0.64,-1.44) ellipse [x radius=0.31, y radius=0.80];

  \fill[hfill, rotate around={12:(0.64,-1.44)}]
      (0.64,-1.44) ellipse [x radius=0.30, y radius=0.80];


  \begin{scope}
    \clip[rounded corners=5pt]
        ( 0.00,  0.22)
     -- (-1.08, -1.24)
     -- ( 1.08, -1.24)
     -- cycle;
    \fill[hint, rotate around={-12:(-0.64,-1.44)}]
        (-0.64,-1.44) ellipse [x radius=0.31, y radius=0.80];
  \end{scope}

  \begin{scope}
    \clip[rounded corners=5pt]
        ( 0.00,  0.22)
     -- (-1.08, -1.24)
     -- ( 1.08, -1.24)
     -- cycle;
    \fill[hint, rotate around={12:(0.64,-1.44)}]
        (0.64,-1.44) ellipse [x radius=0.30, y radius=0.80];
  \end{scope}


  \draw[hborder, rounded corners=5pt]
      ( 0.00,  0.22)
   -- (-1.08, -1.24)
   -- ( 1.08, -1.24)
   -- cycle;

  \draw[hborder, rotate around={-12:(-0.64,-1.44)}]
      (-0.64,-1.44) ellipse [x radius=0.31, y radius=0.80];

  \draw[hborder, rotate around={12:(0.64,-1.44)}]
      (0.64,-1.44) ellipse [x radius=0.30, y radius=0.80];

  \node[elabel] at ( 0.00,-0.75) {$e_0$};
  \node[elabel] at (-0.64,-1.52) {$e_1$};
  \node[elabel] at ( 0.64,-1.52) {$e_2$};

  \node[vertex] (v5) at ( 0.00, -0.20) {$v_5$};
  \node[vertex] (v1) at (-0.56, -0.99) {$v_1$};
  \node[vertex] (v3) at ( 0.56, -0.99) {$v_3$};
  \node[vertex] (v2) at (-0.74, -1.90) {$v_2$};
  \node[vertex] (v4) at ( 0.74, -1.90) {$v_4$};

\end{tikzpicture}}\\
        \resizebox{\linewidth}{!}{\begin{tikzpicture}[
    scale=1.0,
    every node/.style={transform shape},
    vertex/.style={
        circle, draw, fill=white, font=\tiny,
        minimum size=0.33cm, inner sep=0.4pt
    },
    hfill/.style={
        draw=none, fill=gray!22, fill opacity=0.45
    },
    hint/.style={
        draw=none, fill=gray!50, fill opacity=0.28
    },
    hborder/.style={
        draw, line width=0.6pt, fill=none
    },
    elabel/.style={font=\tiny}
]


  \fill[hfill, rounded corners=7pt]
      ( 0.00,  0.22)
   -- (-1.08, -1.24)
   -- ( 1.08, -1.24)
   -- cycle;

  \fill[hfill, rotate around={-8:(-0.6,-1.44)}]
      (-0.6,-1.44) ellipse [x radius=0.29, y radius=0.80];

  \fill[hfill, rotate around={8:(0.6,-1.44)}]
      (0.6,-1.44) ellipse [x radius=0.29, y radius=0.80];

  \fill[hfill]
      (0.00,-1.85) ellipse [x radius=0.89, y radius=0.33];


  \begin{scope}
    \clip[rounded corners=7pt]
        ( 0.00,  0.22)
     -- (-1.08, -1.24)
     -- ( 1.08, -1.24)
     -- cycle;
    \fill[hint, rotate around={-8:(-0.6,-1.44)}]
        (-0.6,-1.44) ellipse [x radius=0.33, y radius=0.80];
  \end{scope}

  \begin{scope}
    \clip[rounded corners=7pt]
        ( 0.00,  0.22)
     -- (-1.08, -1.24)
     -- ( 1.08, -1.24)
     -- cycle;
    \fill[hint, rotate around={8:(0.6,-1.44)}]
        (0.6,-1.44) ellipse [x radius=0.33, y radius=0.80];
  \end{scope}

  \begin{scope}
    \clip (0.00,-1.85) ellipse [x radius=0.89, y radius=0.33];
    \fill[hint, rotate around={-8:(-0.6,-1.44)}]
        (-0.6,-1.44) ellipse [x radius=0.33, y radius=0.80];
  \end{scope}

  \begin{scope}
    \clip (0.00,-1.85) ellipse [x radius=0.89, y radius=0.33];
    \fill[hint, rotate around={8:(0.6,-1.44)}]
        (0.6,-1.44) ellipse [x radius=0.33, y radius=0.80];
  \end{scope}


  \draw[hborder, rounded corners=7pt]
      ( 0.00,  0.22)
   -- (-1.08, -1.24)
   -- ( 1.08, -1.24)
   -- cycle;

  \draw[hborder, rotate around={-8:(-0.6,-1.44)}]
      (-0.6,-1.44) ellipse [x radius=0.33, y radius=0.80];

  \draw[hborder, rotate around={8:(0.6,-1.44)}]
      (0.6,-1.44) ellipse [x radius=0.33, y radius=0.80];

  \draw[hborder]
      (0.00,-1.85) ellipse [x radius=0.89, y radius=0.33];

  \node[elabel] at ( 0.00,-0.75) {$e_0$};
  \node[elabel] at (-0.64,-1.42) {$e_3$};
  \node[elabel] at ( 0.64,-1.42) {$e_1$};
  \node[elabel] at ( 0.00,-1.85) {$e_2$};

  \node[vertex] (v1) at ( 0.00, -0.20) {$v_1$};
  \node[vertex] (v3) at (-0.56, -0.99) {$v_3$};
  \node[vertex] (v2) at ( 0.56, -0.99) {$v_2$};
  \node[vertex] (v5) at (-0.60, -1.85) {$v_5$};
  \node[vertex] (v4) at ( 0.60, -1.85) {$v_4$};

\end{tikzpicture}}\\
        [2pt]
        {\small (a) Hypergraphs}
    \end{minipage}\hfill
    \begin{minipage}[b]{0.30\linewidth}\centering
        \resizebox{\linewidth}{!}{
\begin{tikzpicture}[
    >=Stealth, thick,
    sn/.style={ellipse, draw, fill=gray!12, font=\tiny, inner sep=0.8pt, minimum width=0.85cm, minimum height=0.38cm},
    za/.style={draw, dashed, gray, ->, line width=0.4pt},
    oa/.style={draw, ->, line width=0.5pt}]
  \node[sn] (n0) at (0.00,0.00) {$\emptyset$};
  \node[sn] (n1) at (-0.74,-0.92) {$\emptyset$};
  \node[sn] (n2) at (0.74,-0.92) {$\{e_0,e_1\}$};
  \node[sn] (n3) at (-0.74,-1.84) {$\emptyset$};
  \node[sn] (n4) at (0.74,-1.84) {$\{e_0\}$};
  \node[sn] (n5) at (-1.47,-2.76) {$\emptyset$};
  \node[sn] (n6) at (0.00,-2.76) {$\{e_2\}$};
  \node[sn] (n7) at (1.47,-2.76) {$\{e_0,e_2\}$};
  \node[sn] (n8) at (-0.74,-3.68) {$\emptyset$};
  \node[sn] (n9) at (0.74,-3.68) {$\{e_0\}$};
  \node[sn] (n10) at (0.00,-4.60) {$\emptyset$};
  \draw[za] (n0) to (n1);
  \draw[oa] (n0) to (n2);
  \draw[za, bend right=16] (n1) to (n3);
  \draw[oa, bend left=16] (n1) to (n3);
  \draw[za, bend right=16] (n2) to (n4);
  \draw[oa, bend left=16] (n2) to (n4);
  \draw[za] (n3) to (n5);
  \draw[oa] (n3) to (n6);
  \draw[za] (n4) to (n5);
  \draw[oa] (n4) to (n7);
  \draw[za, bend right=16] (n5) to (n8);
  \draw[oa, bend left=16] (n5) to (n8);
  \draw[za, bend right=16] (n7) to (n9);
  \draw[oa, bend left=16] (n7) to (n9);
  \draw[za, bend right=16] (n6) to (n8);
  \draw[oa, bend left=16] (n6) to (n8);
  \draw[za, bend right=16] (n8) to (n10);
  \draw[oa, bend left=16] (n8) to (n10);
  \draw[za, bend right=16] (n9) to (n10);
  \draw[oa, bend left=16] (n9) to (n10);
\end{tikzpicture}}\\[3pt]
        {\small (b) Flower DD}
    \end{minipage}\hfill
    \begin{minipage}[b]{0.40\linewidth}\centering
        \resizebox{\linewidth}{!}{
\begin{tikzpicture}[
    >=Stealth, thick,
    sn/.style={ellipse, draw, fill=gray!12, font=\tiny, inner sep=0.8pt, minimum width=0.85cm, minimum height=0.38cm},
    za/.style={draw, dashed, gray, ->, line width=0.4pt},
    oa/.style={draw, ->, line width=0.5pt}]
  \node[sn] (n0) at (0.00,0.00) {$\emptyset$};
  \node[sn] (n1) at (-0.74,-0.92) {$\emptyset$};
  \node[sn] (n2) at (0.74,-0.92) {$\{e_0\}$};
  \node[sn] (n3) at (-1.47,-1.84) {$\emptyset$};
  \node[sn] (n4) at (0.00,-1.84) {$\{e_1\}$};
  \node[sn] (n5) at (1.47,-1.84) {$\{e_0,e_1\}$};
  \node[sn] (n6) at (-2.21,-2.76) {$\emptyset$};
  \node[sn] (n7) at (-0.74,-2.76) {$\{e_3\}$};
  \node[sn] (n8) at (0.74,-2.76) {$\{e_1\}$};
  \node[sn] (n9) at (2.21,-2.76) {$\{e_1,e_3\}$};
  \node[sn] (n10) at (-2.21,-3.68) {$\emptyset$};
  \node[sn] (n11) at (-0.74,-3.68) {$\{e_2\}$};
  \node[sn] (n12) at (0.74,-3.68) {$\{e_3\}$};
  \node[sn] (n13) at (2.21,-3.68) {$\{e_2,e_3\}$};
  \node[sn] (n14) at (0.00,-4.60) {$\emptyset$};
  \draw[za] (n0) to (n1);
  \draw[oa] (n0) to (n2);
  \draw[za] (n1) to (n3);
  \draw[oa] (n1) to (n4);
  \draw[za] (n2) to (n3);
  \draw[oa] (n2) to (n5);
  \draw[za] (n3) to (n6);
  \draw[oa] (n3) to (n7);
  \draw[za] (n4) to (n8);
  \draw[oa] (n4) to (n9);
  \draw[za] (n5) to (n8);
  \draw[oa] (n5) to (n9);
  \draw[za] (n6) to (n10);
  \draw[oa] (n6) to (n11);
  \draw[za] (n8) to (n10);
  \draw[oa] (n8) to (n11);
  \draw[za] (n7) to (n12);
  \draw[oa] (n7) to (n13);
  \draw[za] (n9) to (n12);
  \draw[oa] (n9) to (n13);
  \draw[za, bend right=16] (n10) to (n14);
  \draw[oa, bend left=16] (n10) to (n14);
  \draw[za, bend right=16] (n13) to (n14);
  \draw[oa, bend left=16] (n13) to (n14);
  \draw[za, bend right=16] (n11) to (n14);
  \draw[oa, bend left=16] (n11) to (n14);
  \draw[za, bend right=16] (n12) to (n14);
  \draw[oa, bend left=16] (n12) to (n14);
\end{tikzpicture}}\\[3pt]
        {\small (c) Cycle DD}
    \end{minipage}
    \caption{Bounded-width compact DDs for a flower and a cycle hypergraph, described in
    Example~\ref{exa:bounded-compact-dds}: (a) the flower and cycle hypergraphs, (b) the
    flower DD of width $3$, and (c) the cycle DD of width $4$.}
    \label{fig:dd-structures}
\end{figure}





\subsection{Valid antichains, pathwidth,  and incidence treewidth}
\label{subsec:vaw_treewidth_relation}

Researchers have also characterized the hypergraph through \emph{incidence treewidth} to build polynomial-size extended formulations for BPO using a d-DNNF representation~\cite{capelliDelPiaDiGregorio2024kc}. We now show that the valid antichains characterized here are indeed incomparable with the incidence treewidth in Proposition \ref{prop:va-vs-treewidth}. Recall that treewidth is the minimum, over all tree decompositions, of the largest bag size minus one~\cite{diestel2017graph}. The incidence treewidth of a hypergraph is the treewidth of its incidence graph.

Before proving the incomparability, we recall that for ordinary graphs (i.e., rank-two hypergraphs) the valid-antichain width is governed exactly by a classical width parameter. For a graph $G=(V,E)$ and an order $\sigma=(v_1,\ldots,v_n)$ with prefixes $U_i=\{v_1,\ldots,v_i\}$, let $b_i(\sigma)=\left|\{v\in U_i : \exists\, w\in V\setminus U_i \text{ with } \{v,w\}\in E\}\right|$ be the number of \emph{boundary} vertices at $U_i$. The \emph{vertex separation} of $\sigma$ is $\operatorname{vs}_\sigma(G)=\max_{i\in\{1,\ldots,n\}} b_i(\sigma)$, and the \emph{vertex separation number} is $\operatorname{vs}(G)=\min_\sigma \operatorname{vs}_\sigma(G)$. By a classical result, the vertex separation number coincides with pathwidth (i.e., a number that shows how close a graph is to being a path), that is, $\operatorname{vs}(G)=\operatorname{pw}(G)$~\cite{kinnersley1992pathwidth}.

\begin{corollary}\label{cor:vaw-pathwidth}
For every ordinary graph $G=(V,E)$, $\operatorname{VAW}(G)=2^{\operatorname{vs}(G)}=2^{\operatorname{pw}(G)}$.
\end{corollary}

\begin{proof}
Fix an order $\sigma$. At any prefix $U_i$, each active edge has exactly one endpoint in $U_i$ and one outside, so every reduced active edge is a singleton $\{v\}$ with $v$ a boundary vertex, and these singletons are pairwise incomparable. Hence every subset of the $b_i(\sigma)$ boundary singletons is a valid antichain, giving $|\operatorname{VA}_i(\sigma)|=2^{b_i(\sigma)}$. Since $x\mapsto 2^x$ is increasing, $W(\sigma)=\max_i 2^{b_i(\sigma)}=2^{\operatorname{vs}_\sigma(G)}$, and minimizing over $\sigma$ yields $\operatorname{VAW}(G)=2^{\operatorname{vs}(G)}$. The identity $\operatorname{vs}(G)=\operatorname{pw}(G)$~\cite{kinnersley1992pathwidth} gives the last equality. \hfill$\square$
\end{proof}

We note that this identity is specific to rank-two hypergraphs; for hypergraphs of higher rank, valid-antichain width departs from pathwidth, as the converse direction below shows.

\begin{proposition} \label{prop:va-vs-treewidth}
The valid antichain characterization is incomparable to the incidence treewidth.
\end{proposition}

\begin{proof}
We first show that bounded incidence treewidth does not imply bounded VAW. Consider $T_h=(V,E)$, the complete binary tree of height $h$, where each edge has exactly two vertices. Since $T_h$ is a tree, its treewidth is one. Moreover, the incidence graph of $T_h$ is obtained by subdividing every edge of $T_h$ once, so each original edge becomes an incidence node adjacent to its two endpoints; the incidence graph is therefore also a tree and has treewidth one. On the other hand, complete binary trees have unbounded pathwidth as $h$ grows~\cite{diestel2017graph}. Hence, by Corollary~\ref{cor:vaw-pathwidth}, $\operatorname{VAW}(T_h)=2^{\operatorname{pw}(T_h)}$ is also unbounded. Thus, the family $\{T_h\}_{h\ge 1}$ has bounded incidence treewidth but unbounded VAW.

We now show the converse also fails. Let $A=\{a_1,\ldots,a_p\}$ and
$B=\{b_1,\ldots,b_q\}$ be sets of vertices, and define a hypergraph $G=(V,E)$ with $V=A\cup B$, and $E=\{A\cup\{b\}: b\in B\}$. Order all vertices of $A$ first and then all vertices of $B$. At any prefix $U_i$ of this order, all active hyperedges have the same reduced active hyperedge. Therefore, there are at most two valid antichains, the empty one and the singleton containing that common
reduced hyperedge, so $|\operatorname{VA}_i(\sigma)| \le 2$. On the other hand, the incidence graph contains a complete bipartite graph between $A$ and the hyperedges, so its incidence treewidth is unbounded as the number of vertices grows.\hfill$\square$
\end{proof}




\section{Target cuts over compact DDs}
\label{sec:dd_separation}

We now present the cut generator linear program (CGLP) that creates valid inequalities using our compact DD encoding introduced in Section \ref{sec:dds_for_bpo}. As previously mentioned, there are several cut generation procedures specifically designed to separate a fractional point from a convex set encoded as a DD \cite{beckerBehleEisenbrandWimmer2005,behle2007,castroCireBeck2022cutLift,davarniaVanHoeve2021outer}. In this work, we opt to use DD-based target cuts~\cite{tjandraatmadjaVanHoeve2019target}, due to their simple implementation and their guarantee of facet-defining inequalities.

Broadly, given a fractional point $\bar \bz$ to separate and a fixed interior point $\bomega$ of the polytope represented by the DD, the target cuts procedure~\cite{buchheimLiersOswald2008localcuts} searches over the translated polar of that polytope for the valid inequality of maximum violation in the direction from $\bomega$ to $\bar \bz$. The correspondence between DDs and polarity allows this polar to be optimized compactly over the DD network flow model~\cite{tjandraatmadjaVanHoeve2019target}, and whenever the returned solution is an extreme point of the polar, the associated inequality is facet-defining for the represented polytope. As with the extended formulation of Section~\ref{sec:dds_for_bpo}, the CGLP of~\cite{tjandraatmadjaVanHoeve2019target}  assumes that every variable appearing in the cut is represented explicitly in the DD. Our compact DD, however, branches only on the vertex variables $\bz_v$ and recovers the hyperedge variables $\bz_e$ implicitly along its states, so the standard target-cut CGLP does not apply directly. We now present the extension to our compact DD encoding.

In what follows, let $B=(N,A)$ be the compact DD encoding 
for a hypergraph $G=(V,E)$, so $B$ is exact for $\SG$ and therefore represents $MP_G$. Recall from Section~\ref{sec:dds_for_bpo} that, for each arc $a\in A$ leaving layer $N_i$ ($i \in \{1,\ldots,n\}$), $\var(a)=v_i$ is the associated vertex, $\val(a)\in\{0,1\}$ its assigned value, and $\E_1(a)$ the set of hyperedges that take value one on arc $a$. Throughout, let $\bomega\in\operatorname{int}(MP_G)$ be a fixed interior point; since $MP_G$ is full-dimensional~\cite{del2017polyhedral} and $\bomega$ lies in its interior, the translated polar $(MP_G-\bomega)^\circ$ is bounded and the CGLP below attains its optimum.

Consider continuous variables $\bu\in\mathbb{R}^{V\cup E}$ and $\btheta\in\mathbb{R}^N$ and define arc sets associated with $0$ and $1$ assignments as $A^1=\{a\in A:\val(a)=1\}$ and $
A^0=\{a\in A:\val(a)=0\}$, respectively. Given a fractional point $\bar \bz=(\bar \bz_v,\bar \bz_e)\in\mathbb{R}^{V\cup E}$, the target-cut CGLP for our compact DD encoding is as follows, where $r$ and $t$ denote the root and terminal nodes of $B$, respectively:
\begin{subequations}\label{model:target-cut}
\begin{align}
\max \quad & \bu^\top(\bar \bz-\bomega)\\
\text{s.t.}\quad
& \theta_j
\le \theta_i
- u_{\var(a)}
- \sum_{e\in\E_1(a)} u_e
&& \forall a=(i,j)\in A^1,
\label{eq:tc-bdd-one-arc}\\
& \theta_j \le \theta_i
&& \forall a=(i,j)\in A^0,
\label{eq:tc-bdd-zero-arc}\\
& \theta_r = 1 + \bu^\top\bomega,
&&
\label{eq:tc-bdd-root}\\
& \theta_t = 0.
&&
\label{eq:tc-bdd-terminal}
\end{align}
\end{subequations}

Constraints \eqref{eq:tc-bdd-one-arc}--\eqref{eq:tc-bdd-root} encode the translated
polar. Note that an arc $a\in A$ with $\val(a)=1$ (i.e., inequalities \eqref{eq:tc-bdd-one-arc}) contributes the coefficient of the assigned vertex and the coefficients of all hyperedges assigned to one on that arc; while the remaining arcs contribute nothing. 
Summing the arc contributions along an $r$--$t$ path recovers $\bu^\top \bz$ for point $\bz$, so the model enforces $\bu^\top(\bz-\bomega)\le 1$ for every point of $MP_G$.

Model~\eqref{model:target-cut} has $O(|A|)$ constraints and $|V|+|E|+|N|$ variables, so its size is polynomial in the size of $B$; together with the structural results of Section~\ref{sec:bdd_size_structures}, a bounded valid-antichain width yields polynomial-time separation. As in the target-cut framework, if the optimal value of \eqref{model:target-cut} is at most $1$, then the DD yields no violated target cut. Otherwise, the resulting inequality $ \bu^\top \bz \le 1 + \bu^\top\bomega $ is valid for the represented polytope $MP_G$ and cuts off $\bar\bz$.

Since multilinear polytopes are full-dimensional for every hypergraph \cite{del2017polyhedral}, the translated polar has the usual facet correspondence~\cite{buchheimLiersOswald2008localcuts}. If the solution returned by \eqref{model:target-cut} is an extreme point of $(MP_G-\bomega)^\circ$, then the associated inequality corresponds to a facet of the represented polytope $MP_G$. In particular, a unique optimum is sufficient for this conclusion. If the optimum is degenerate, the returned inequality remains valid, but facetness is not automatic unless an extreme optimal solution is selected.
\section{Local supports and expansions for cut generation}
\label{sec:support_expansion}

As mentioned in Section \ref{sec:bdd_size_structures}, the compact DD encoding can be exponential in size for generic BPO instances. Thus, constructing such a DD and generating cuts with it can be impractical. With that in mind, we develop three different procedures that build compact DDs for portions of the problem (i.e., local supports for a hypergraph $ G = (V,E)$) with the aim of returning strong valid inequalities for the original problem. 

Formally, a \emph{local support} for $G=(V,E)$ is a hypergraph $G_S=(V_S,E_S)$ chosen from $G$, typically with $V_S\subseteq V$ and $E_S\subseteq E$, which are also known as \emph{partial hypergraphs}~\cite{del2017polyhedral}. If $E_S$ has all the hyperedges that are subsets of $V_S$ (i.e., $E_S=\{e\in E : e\subseteq V_S\}$), then we say that $G_S$ is a \emph{section hypergraph} of $G$ induced by $V_S$. Once a support $G_S$ has been chosen, we build a compact DD for it and use the target-cut CGLP introduced in Section~\ref{sec:dd_separation} to create valid inequalities. The generated cut can be lifted back to the full variable
space by \emph{zero lifting}~\cite{del2017polyhedral} (i.e., setting coefficients of variables outside $V_S\cup E_S$ to zero).

We now describe three alternatives for creating this local support. To summarize, the \emph{fixed-support} strategy considers supports from known valid inequalities in the literature (i.e., two-links, flowers, and $\beta$-cycles) and serves as a validation benchmark. The \emph{expanded-support} strategy expands over fixed-supports by considering additional edges with the goal of exposing nearby interactions that the literature support alone omits. Lastly, the \emph{partition-support} strategy partitions the set of vertices and builds section hypergraphs induced by each partition to leverage existing theoretical results. 


\paragraph{Fixed-support strategy.} For a given fractional point $\bar \bz$, this strategy builds compact DDs based on a known violated valid inequality. Specifically, we heuristically search for violated two-link, flower, and $\beta$-cycle inequalities and create a DD for the corresponding local support for any such violated cuts. To do so, we work in the edge-overlap graph of $G$, whose nodes are the hyperedges and adjacency is defined by nonempty overlap. Starting from a set of sampled edges, we build small-overlap neighborhoods and mine candidate local supports within those neighborhoods. A candidate's support is passed to the DD separator only if its associated literature inequality is violated at the current relaxation point $\bar \bz$. We note that, given its nature, this strategy is merely a validation benchmark for our DD-based CGLP.




\paragraph{Expanded-support strategy.} This strategy is an extension of our fixed-support approach, with the aim of exposing nearby interactions that are not visible from the support of simple structures (i.e., two-links, flowers, and $\beta$-cycles) alone. The strategy is motivated by the inducing-support conditions that appear in zero-lifting results for multilinear polytopes~\cite{del2017polyhedral}: supports that omit overlapping hyperedges may miss relevant overlap information. Expansion is therefore a heuristic approach for incorporating more of this local overlap structure (i.e., adjacent vertices and hyperedges) into the local support, with the goal of obtaining stronger cuts.

\begin{algorithm}[t]
\caption{\textsc{ExpandedSupportStrategy}}
\label{alg:expanded_support_mode}
\begin{algorithmic}[1]
\Require Hypergraph $G=(V,E)$, support $G_S=(V_S, E_S)$, hyperedge cap $M$
\While{$|E_S|<M$ and there exists an $e'\in E\setminus E_S$ that satisfies $|e'\cap V_S|\ge 2$} \label{line:support:candidate}
    \State Choose an eligible $e'$ of maximum overlap with $E_S$, that is,  
    \[
    e' :=\argmax_{e'' \in E}
    \left\{\sum_{e\in E_S}|e''\cap e| \ :\ e'' \in E\setminus E_S,\ |e''\cap V_S|\ge 2  \right\},
    \] \label{line:support:max_overlap}
    \State Update $E_S \gets E_S \cup \{e'\}$ and $V_S\gets \bigcup_{e\in E_S} e$\label{line:support:update_support}
\EndWhile
\State \Return expanded local support $G_S=(V_S, E_S)$ 
\end{algorithmic}
\end{algorithm}

Algorithm~\ref{alg:expanded_support_mode} summarizes the expanded-support strategy. The procedure receives an initially violated local support $G_S=(V_S, E_S)$ of $G$ (found by the heuristic used in the previous strategy) and a maximum number of hyperedges $M$ that are allowed in the expanded support. It then iteratively selects an eligible hyperedge that overlaps the current vertex set and maximizes overlap with the current hyperedges (lines \ref{line:support:candidate}--\ref{line:support:max_overlap})
and updates the local support accordingly (line \ref{line:support:update_support}). 
Once the hyperedge cap is reached, we create the DD on the resulting support, find a valid inequality, and zero-lift it to the original variable space.


\paragraph{Partition-support strategy.} Unlike the other two alternatives, this strategy builds a set of local supports based on a given vertex partition, without relying on the simple local structures from known valid inequalities. For any fractional point $\bar \bz$, this strategy iterates over the pre-defined set of local supports to create one DD-based valid inequality for each local support $G_S$ where $\bar \bz \notin MP_{G_S}$. The strategy is designed such that we can: (i) test whether generic local supports can render useful cuts, and (ii) leverage the theoretical results for section hypergraphs to obtain facet-defining inequalities under suitable conditions. 

Algorithm \ref{alg:partition_support_mode} summarizes the procedure for creating the local supports. First, we arbitrarily partition the set of vertices $V$ into $H$ disjoint sets $\{V_1,\ldots,V_H\}$. Then, for each vertex group $V_h$ ($h \in \{1,\ldots,H\}$), we create a local support following the \emph{contained-edge rule}, that is, consider all hyperedges that are contained in $V_h$ (i.e., $E_h=\{e\in E:e\subseteq V_h\}$). Then, by construction, each local support $G_h=(V_h,E_h)$ is a
section hypergraph of $G$~\cite{del2017polyhedral}. 

The main benefit of building a section hypergraph $G_h=(V_h,E_h)$ is that, if $V_h$ is \emph{inducing} (i.e., every $e\in E$ with $|e\cap V_h|\ge 2$ satisfies $e\cap V_h\in E$) and the resulting inequality is facet-defining for $MP_{G_h}$, then its zero-lifting is facet-defining for $MP_G$ whenever it is maximal~\cite{del2017polyhedral}. Since our DD-based CGLP \eqref{model:target-cut} can return facets of the local polytope $MP_{G_h}$ (Section~\ref{sec:dd_separation}), section-hypergraph supports are precisely the setting in which these facets can be zero-lifted to facets of the full instance.

\begin{algorithm}[t]
\caption{\textsc{PartitionSupportStrategy}}
\label{alg:partition_support_mode}
\begin{algorithmic}[1]
\Require Hypergraph $G=(V,E)$, and  vertex partition $\{V_1,\ldots,V_H\}$
\For{$h \in \{1,\ldots,H\}$}
    \State Select edges with the contained-edge rule: $E_h\gets \{e\in E : e\subseteq V_h\}$
    \State Create local support $G_h =(V_h, E_h)$
\EndFor
\State \Return set of local supports $\{G_1,\ldots, G_H\}$
\end{algorithmic}
\end{algorithm}





\bigskip

We note that in our implementation, the fixed-support and expanded-support strategies are repeatedly called at the fractional point $\bar \bz$. Thus, local supports are selected dynamically, and their corresponding DDs are created in each iteration. In contrast, the partition-support strategy builds the set of local supports and their corresponding DDs at the beginning of the optimization process and reuses them whenever we want to separate a fractional point. 

\section{Dimension and facet audit for DD cuts} 
\label{sec:dimension_facet_audit}

We now present a novel approach for computing the dimension of the face induced by a valid inequality for the polytope encoded by a DD. We present the procedure in general terms and then extend it to the compact DD encoding introduced in Section \ref{sec:dds_for_bpo}. We then show how to use this procedure to obtain facet-defining certificates for our cuts, which we employ in Section \ref{sec:experimental_design} as a quality measurement.

We now present the notation used throughout this section. Consider a set $Q\subseteq\{0,1\}^n$ and let $\dim(Q)$ be its dimension. We say that an inequality $\bpi^\top \bz\le \pi_0$ is \emph{valid} for $\conv(Q)$ if it holds for every $\bz\in Q$, and its associated face is $F(\bpi)=\{\bz\in \conv(Q) : \bpi^\top \bz=\pi_0\}$. If $F(\bpi)\neq\emptyset$ and $\dim(F(\bpi))=\dim(Q)-1$, then $F(\bpi)$ is a \emph{facet} of $\conv(Q)$.

In what follows, we introduce vector notation to represent the value associated with an arc, which is amenable to our DD-based dimension certificate procedure. Let $B=(N,A)$ be a DD representing $Q$. Recall from  Section~\ref{sec:dds_for_bpo} that  $\var(a) \in V$ and $\val(a)\in\{0,1\}$ are the vertex and value associated with arc $a \in A$, respectively. Then, define the vector contribution of an arc $a \in A$ as $\chi(a) = \be^{\var(a)}$ if $\val(a)=1$ and $\chi(a) = \bzero$ otherwise, where $\be^{q}$ denotes the q-th canonical (unit) vector ($q \leq n$) and $\bzero$ the zero vector of $\mathbb{R}^n$. Then, for a valid inequality $\bpi^\top \bz\le \pi_0$ of $Q$, we define the \emph{scalar arc length} of arc $a \in A$ as $\ell_\pi(a)= \bpi^\top\chi(a)$, where  $\ell_\pi(a)= \pi_{\var(a)}$ if $\val(a)=1$ and  $\ell_\pi(a)= 0$ otherwise. 

Throughout this section, lowercase $p$ denotes a \emph{partial path} (i.e., a sequence of arcs) over $B=(N,A)$ from root node $r$ to a given node $u \in N$, while uppercase $P$ denotes a complete \rt\ path. Then, for any partial path $p$, define $z(p)=\sum_{a\in p}\chi(a)$ as its vector encoding, i.e., a vector that encodes all the decisions for arcs $a \in p$ but sets to zero all decisions after the last arc represented in $p$. Note that if $p = P$ (i.e.,  an \rt\ path) then $z(P)$ is the binary point encoded by $P$.

\subsection{Tight arcs and dimension certificate}
\label{subsec:tight_arc_identification}

We now present our DD-based dimension certificate for a DD $B=(N,A)$ that exactly represents $Q$ and a valid inequality $\bpi^\top\bz \leq \pi_0$ with $F(\bpi) \neq \emptyset$. The overall idea is to work over the arcs and paths in $B$ where $\bpi^\top\bz = \pi_0$ and then collect a set of difference vectors whose rank equals the dimension of $F(\bpi)$.

First, we identify the arcs in $B$ that are associated with the binary points in $F(\bpi)$, which we call \emph{tight arcs}. Consider top-down and bottom-up longest-path values for any node $u \in N$ as $ L^\uparrow(\bpi,u) =\max\left\{\sum_{a\in p}\ell_\pi(a): p \text{ is an } r\text{--}u \text{ path}\right\}$ and $ L^\downarrow(\bpi,u)= \max\left\{\sum_{a\in p}\ell_\pi(a): p \text{ is a } u\text{--}t \text{ path}\right\} $, respectively \cite{tjandraatmadjaVanHoeve2019target,castroCireBeck2022cutLift}. 
Then, we say that an arc $a=(u,u')$ is a tight arc if it belongs to at least one tight \rt\ path, that is, if $L^\uparrow(\bpi,u)+\ell_\pi(a)+L^\downarrow(\bpi,u')=\pi_0$, so that the longest path traversing arc $a$ has length $\pi_0$. For each stage $i\in\{1,\ldots,n\}$, let $A^\pi_i$ be the set of outgoing arcs in layer $N_i$ that are tight, and let $A^\pi=\bigcup_{i =1}^n A^\pi_i$ be the set of tight arcs. Lastly, let $\Paths_\pi(r,u)$ denote the set of $r$--$u$ paths that are prefixes of at least one tight \rt\ path.

\begin{algorithm}[t]
\caption{\textsc{DD-DimensionCertificate}}
\label{alg:dd_dimension_audit}
\begin{algorithmic}[1]
\Require Exact DD $B=(N,A)$, valid inequality $\bpi^\top \bz\le \pi_0$ with $F(\bpi)\neq \emptyset$ 
\Ensure Dimension of  $F(\bpi)$
\State Create the sub-diagram $B^\pi= (N^\pi, A^\pi)$ induced by $A^\pi$ \label{line:dim:create_dd}
\State Set $\rep[u]\gets$ undefined for all $u\in N^\pi$,  $\rep[r]\gets \bzero$, and $\M\gets\emptyset$
\For{each layer $i \in \{1, \ldots, n\}$}
    \For{each arc $a=(u,u')\in A^\pi_i$}
        \State $\operatorname{cand}\gets \rep[u]+\chi(a)$\label{line:dim:create_candidate}
        \If{$\rep[u']$ is undefined} 
            \State $\rep[u']\gets \operatorname{cand}$ \label{line:dim:set_candidate}
        \Else
            \State add $\operatorname{cand}-\rep[u']$ to $\M$ \label{line:dim:add_to_M}
        \EndIf
    \EndFor
\EndFor
\State \Return $\operatorname{rank}(\M)$
\end{algorithmic}
\end{algorithm}

Our DD-based dimension certificate algorithm, summarized in Algorithm \ref{alg:dd_dimension_audit}, uses the sub-diagram $B^\pi= (N^\pi, A^\pi)$ of $B$ induced by the tight arcs $A^\pi$ (i.e., $N^\pi =\{r\}\cup \{ u \in N: \ \exists a \in A^\pi \text{ s.t. $a\in \Ain(u)$ } \}$) (line \ref{line:dim:create_dd}). The algorithm traverses $B^\pi$ starting at the root $r$ and maintains, for each node $u\in N^\pi$, a representative vector $\rep[u]$ equal to the encoding $z(p_u)$ of a single tight prefix path from $r$ to $u$. It starts from $\rep[r]=\bzero$ and processes the tight arcs layer by layer, so that $\rep[u]$ is already available when the outgoing arcs of $u$ are examined. When a tight arc $a=(u,u')$ reaches $u'$ for the first time, the algorithm fixes $\rep[u']=\rep[u]+\chi(a)$, extending the prefix chosen for $u$ along $a$ (lines \ref{line:dim:create_candidate}--\ref{line:dim:set_candidate}); these first-visit arcs form a spanning tree of $B^\pi$ rooted at $r$. Every other tight arc into an already-represented node $u'$ yields an alternative tight prefix reaching $u'$, and the discrepancy $\operatorname{cand}-\rep[u']=(\rep[u]+\chi(a))-\rep[u']$ between the two encodings is added to the matrix $\M$ (lines \ref{line:dim:create_candidate} and \ref{line:dim:add_to_M}). As shown by Lemma \ref{lem:dd_dimension_invariant} and Theorem \ref{thm:dimension_certificate}, each such vector is a difference of two tight \rt\ path vectors, and these differences span their entire linear space; hence the algorithm returns $\operatorname{rank}(\M)$, which equals the dimension of $F(\bpi)$. The proof uses the concatenation operation $\circ$ to augment a path: if $p' = (a_1,\ldots, a_k)$ is a partial path with $k$ arcs, then  $p = p' \circ a = (a_1,\ldots, a_k, a)$.

\begin{lemma}
\label{lem:dd_dimension_invariant}
In Algorithm~\ref{alg:dd_dimension_audit}, every representative vector $\rep[u]$, with $u \in N^\pi$, equals $z(p_u)$ for some partial path $p_u\in\Paths_\pi(r,u)$. Moreover, for every node $u \in N^\pi$ and partial path $p\in\Paths_\pi(r,u)$, $z(p)-\rep[u]\in\operatorname{span}(\M)$.
\end{lemma}

\begin{proof}
We proceed by induction over the layers. The claim is immediate at the root $r$, where
$\rep[r]=\bzero$, which corresponds to $z(r)$ and $\bzero \in \operatorname{span}(\M)$ for any matrix $\M$. Then, consider the hypothesis true for all nodes in layer $N^\pi_i$, and take any arc $a=(u,u')\in A^\pi$ with $ u \in N^\pi_i$,  $p=p'\circ a$, with $p'\in\Paths_\pi(r,u)$. Then, $z(p)=z(p')+\chi(a)$, and, by the induction hypothesis, $z(p')-\rep[u]\in\operatorname{span}(\M)$.

If $\rep[u']$ is undefined, the algorithm sets $\rep[u']=\rep[u]+\chi(a)=z(p_u\circ a)$, with $p_u\circ a\in\Paths_\pi(r,u')$, which establishes the first claim for $u'$. Hence 
$z(p)-\rep[u']
=
z(p')+\chi(a)-\bigl(\rep[u]+\chi(a)\bigr)
=
z(p')-\rep[u]$, which belongs to $\operatorname{span}(\M)$ by the induction hypothesis.

If $\rep[u']$ is already defined, the algorithm adds $\delta_a=\rep[u]+\chi(a)-\rep[u'] $ to $\M$. Therefore,
\[
\begin{aligned}
z(p)-&\rep[u']
=z(p')+\chi(a)-\rep[u'] \\
&=\bigl(z(p')-\rep[u]\bigr)
+\bigl(\rep[u]+\chi(a)-\rep[u']\bigr) =\bigl(z(p')-\rep[u]\bigr)+\delta_a,
\end{aligned}
\]
which belongs to $\operatorname{span}(\M)$ since it is the sum of two vectors in $\operatorname{span}(\M)$. \hfill $\square$
\end{proof}

\begin{theorem} \label{thm:dimension_certificate}
Algorithm~\ref{alg:dd_dimension_audit} returns the dimension of $F(\bpi)$ over $\conv(Q)$.
\end{theorem}

\begin{proof}

Since DD $B$ exactly represents $Q$, the extreme points of $F(\bpi)$ are exactly the vectors $z(P)$ for tight \rt\ paths $P \in \Paths_\pi(r,t)$; thus $F(\bpi)=\conv(\{z(P):\ P \in \Paths_\pi(r,t) \})$. By Lemma~\ref{lem:dd_dimension_invariant}, $z(P)-\rep[t]\in\operatorname{span}(\M)$ for every such $P$. Therefore, we conclude that $\dim(F(\bpi)) \leq \operatorname{rank}(\M)$.

For the reverse inclusion, we show that every vector added to $\M$ is the difference of two tight \rt\ path vectors. Consider a vector $\rep[u]+\chi(a)-\rep[u']$ added to $\M$ while processing an arc $a=(u,u')$. By Lemma~\ref{lem:dd_dimension_invariant}, $\rep[u]=z(p_u)$ for some $p_u\in\Paths_\pi(r,u)$ and $\rep[u']=z(p_{u'})$ for some $p_{u'}\in\Paths_\pi(r,u')$. Now choose any tight $u'$--$t$ path $q$ and consider $P' = p_u\circ a\circ q$ and $P = p_{u'}\circ q$. 

Recall that any path in $\Paths_\pi(r,w)$ has length $L^\uparrow(\bpi,w)$, since $w\in N^\pi$ lies on a tight path. Therefore, both $P'$ and $P$ are tight \rt\ paths because $p_u$ and $p_{u'}$ have lengths $L^\uparrow(\bpi,u)$ and $L^\uparrow(\bpi,u')$, $q$ is tight, and $a$ is a tight arc. Then,
\[
\begin{aligned}
z(P')-z(P)
&=z(p_u)+\chi(a)+z(q)-z(p_{u'})-z(q)\\
&=\rep[u]+\chi(a)-\rep[u'].
\end{aligned}
\]
The common suffix cancels, so the vector added to $\M$ by Algorithm~\ref{alg:dd_dimension_audit} (line \ref{line:dim:add_to_M}) is a difference of two tight \rt\ path vectors. Hence $\operatorname{span}(\M)$ is contained in the linear span of all differences of $z(P)$ for $P \in \Paths_\pi(r,t)$. Together with the first inclusion, this gives $\dim(F(\bpi))=\operatorname{rank}(\M)$. \hfill $\square$
\end{proof}

We now adapt Algorithm \ref{alg:dd_dimension_audit} to the compact DD encoding for BPO. Specifically, given a compact DD $B=(N,A)$ for a hypergraph $G=(V,E)$, we need to adjust the definitions of $\chi(a)$ for $a\in A$ to account for the hyperedge variables. Recall, from 
Section~\ref{sec:dds_for_bpo}, that
$\E_1(a)$ is the set of hyperedge variables $\bz_e$ assigned value one on arc $a\in A$. Then, for an arc $a\in A$, we redefine $\chi(a) = \be^{\var(a)}+\sum_{f\in\mathcal{E}_1(a)}\be^{f}$ when  $\val(a)=1$ and $\chi(a)=\bzero$ otherwise. Thus, $\chi(a)$ is a vector of dimension $|V|+|E|$ that accounts for both vertex and hyperedge variables assigned in such an arc. 
The rest of the definitions and steps of the algorithm (e.g., $\ell_\pi(a)$) follow from the new  $\chi(a)$ definition. Note that since $MP_G$ has full dimension (i.e., $\dim(MP_G)=|V|+|E|$)~\cite{del2017polyhedral}, any valid inequality is facet-defining for $MP_G$ if Algorithm~\ref{alg:dd_dimension_audit} returns $|V|+|E|-1$.

\subsection{The $I_K$ certificate and zero lifting}
\label{subsec:local_facet_zero_lifting}

Algorithm \ref{alg:dd_dimension_audit} provides a global facet test for $MP_G$ when applied to the hypergraph instance $ G = (V,E)$; however, it is impractical for large instances. We therefore use the certificate over local supports used for cut generation (see Section \ref{sec:dd_separation}), and the zero-lifting theory~\cite{del2017polyhedral}. The resulting test is conservative: if it succeeds, the cut is certified to be a global facet of $MP_G$; if it fails, the global status is unknown.

Del Pia and Khajavirad~\cite{del2017polyhedral} show that when a local support is an inducing section hypergraph, a zero-lifted local facet remains facet-defining for the larger multilinear polytope if and only if it is maximal for the larger hypergraph. Their Corollary~6 can be used as a conservative certificate even when the original support is not inducing by completing the support with the required intersection hyperedges, auditing the cut on this completed support, and then checking maximality.

For a valid inequality $\bpi^\top \bz\le \pi_0$, consider $K=G(\bpi)$  to be its support hypergraph, in which its hyperedge set is $E_K=\{e\in E : \pi_e\ne 0\}$, and its vertex set is $V_K=
\{v\in V : \pi_v\ne 0\}
\cup
\{v\in V : \exists e\in E \text{ with } v\in e,\ \pi_e\ne 0\}$.
Let $G_K$ be the section hypergraph of the full instance $G=(V,E)$ induced by $V_K$, that is, $V_{G_K}=V_K$ and $E_{G_K}=\{e\in E : e\subseteq V_K\}$. Then, define the intersection completion hypergraph as $I_K=(V_K,E_{I_K})$ with hyperedge set  $E_{I_K}= E_{G_K}\cup\{e\cap V_K : e\in E,\ |e\cap V_K|\ge 2\}$, where duplicate hyperedges are removed and hyperedges in $E_{I_K}\setminus E_{G_K}$ correspond to auxiliary certificate hyperedges. To audit the cut on $I_K$, we zero-extend $\bpi$ by assigning a coefficient of zero to every auxiliary hyperedge in $E_{I_K}\setminus E_K$.

The remaining ingredient is the maximality check. Maximality is violated if there exists a decomposition obtained by multiplying the local inequality by a monomial involving variables outside the selected subhypergraph and then linearizing the resulting expressions. This test is performed first, before the remaining certificate checks, using the finite criterion of Del Pia and Khajavirad~\cite{del2017polyhedral}, based on their Proposition~7 and Lemma~2. If the maximality test fails, the candidate inequality is rejected as a facet of the full multilinear polytope; otherwise, the audit proceeds to the remaining certificate checks.

\begin{algorithm}[t]
\caption{\textsc{GlobalFacetAuditViaIK}}
\label{alg:ik_certificate_audit}
\begin{algorithmic}[1]
\Require Full hypergraph $G=(V,E)$, cut $\bpi^\top \bz\le \pi_0$
\Ensure Certification status for global facetness
\If{Cut is non-maximal}\label{line:ik:maximality}
    \State \Return non-maximal; classify the cut as a global nonfacet
\EndIf
\State Build the compact DD $B_{I_K}$ over $I_K$\label{line:ik:build-ik}
\State $d \gets \textsc{DD-DimensionCertificate}(B_{I_K},\bpi^\top \bz\le \pi_0)$\label{line:ik:audit-ik}
\If{$d = |V_K| + |E_{I_K}| - 1$}\label{line:ik:rank-ik}
    \State \Return certified global facet by the $I_K$ certificate
\EndIf
\If{Exact global DD is computationally feasible}\label{line:ik:global-feasible}
    \State Build the compact exact global DD $B_{G}$\label{line:ik:build-global}
    \State $d_G \gets \textsc{DD-DimensionCertificate}(B_G,\bpi^\top \bz\le \pi_0)$\label{line:ik:audit-global}
    \If{$d_G = |V| + |E| - 1$}\label{line:ik:rank-global}
        \State \Return certified global facet by direct global audit
    \Else
        \State \Return certified global nonfacet by direct global audit
    \EndIf
\EndIf
\State \Return unknown global status
\end{algorithmic}
\end{algorithm}

Algorithm~\ref{alg:ik_certificate_audit} presents the complete procedure. It first checks the maximality (line~\ref{line:ik:maximality}); if the inequality is non-maximal, it is classified as a global nonfacet for this audit. If maximality passes, the algorithm builds the DD $B_{I_K}$ and runs Algorithm~\ref{alg:dd_dimension_audit} on it (lines~\ref{line:ik:build-ik}--\ref{line:ik:audit-ik}). The rank test (line~\ref{line:ik:rank-ik}) checks the dimension of the cut's face; if it equals $|V_K|+|E_{I_K}|-1$, the cut is certified as a global facet by the $I_K$ certificate. If this local certificate fails, the algorithm attempts a direct global audit when the exact DD for the complete problem can be built. The final rank test (line~\ref{line:ik:rank-global})  certifies either global facetness or global nonfacetness. If the global DD is not computationally feasible, the global status remains unknown.
\section{Computational results}
\label{sec:experimental_design}

This section presents the empirical evaluation of our DD-based cut generation procedures. We first focus on the quality of the generated cuts at the root node, in terms of gap reduction and computational efficiency, and compare with alternatives in the literature. Then, we test the efficiency of our procedures for solving instances from the literature to optimality using a B\&B mechanism with cuts added at the root node, and compare them with existing alternatives. 

\subsection{Experimental setup}

We consider two datasets commonly used in the literature: LABS and VISION.  The LABS instances are the low-autocorrelation binary sequence instances distributed in POLIP and MINLPLib~\cite{polip,minlplib,delPiaWalter2024simple}. Specifically, we consider instances with fewer than 5000 multilinear terms and discard the \texttt{autocorr\_bern20-03} instance (which solves at the root node), for a total of 22 instances. The VISION benchmark consists of the 45 image-restoration instances introduced in the BPO literature~\cite{delPiaKhajaviradSahinidis2020impact}, with image sizes $10\times 10$, $10\times 15$, and $15\times 15$.

Our experimental comparison considers four families of cut-generation routines (i.e., LT, FS, EX, and PT), all of which are applied at the root node, as follows. LT denotes the literature-support baseline, which separates the violated two-link, flower, and odd $\beta$-cycle inequalities of length 3 found by well-known separators. FS  corresponds to the DD separation on the same fixed supports used by LT to test whether the DD separator can recover comparable or stronger cuts from those supports. EX$M$ denotes DD separation on expanded literature supports, where $M\in \{4,5,6,7,8\}$ is the hyperedge cap in Algorithm~\ref{alg:expanded_support_mode}. Lastly, PT$q$ is the partition-support strategy, where $q\in \{0,1,2,3,4\}$ is the partition size parameter. Specifically, the approach considers $r=\max_{e\in E}|e|$ to be the rank of the instance hypergraph, and each PT$q$ strategy considers a vertex partition generated with target size $r+q$, inducing a section hypergraph as described in  Algorithm~\ref{alg:partition_support_mode}. 

All DD-based CGLPs use the same interior point $\bomega$.
For each local support hypergraph $G_S=(V_S,E_S)$, we set
$
\omega_v=\frac{1}{2}
$
for each $v\in V_S$ and
$
\omega_e=2^{-|e|}
$
for each $e\in E_S$.

Our methodology is implemented in Python 3.11.15 using Gurobi 13.0.1, with a 3600\,s time limit across all experiments. We use a single-thread (i.e., \texttt{Threads=1}), set \texttt{PreCrush=1} to allow user cuts to be translated through presolve, and deactivate native cuts in Gurobi (i.e., \texttt{Cuts=0}, \texttt{MIPSepCuts=0}, and \texttt{CutPasses=0}) to avoid interference of external cuts in our comparisons. The experiments were run on the \textit{Ingeniería UC Cluster} using SLURM, equipped with an AMD EPYC 9354 processor and 256\,GB of RAM.

A separation round solves the current LP relaxation, searches for local supports under
the selected strategy, and adds at most $\lceil 0.05|E|\rceil$ violated cuts. For
neighborhood-based strategies, the initial neighborhood contains 300 hyperedges and is
expanded by 30 hyperedges whenever no candidate support is found, until the search budget is exhausted.

In all experiments, cuts are added to the root node based on the following criteria. Separation stops when a round finds no cuts; for PT, up to five consecutive no-cut rounds are allowed before stopping. Separation also stops if the average LP improvement per cut over the last 50 rounds is less than $10\%$ of the historical average LP improvement per cut, or if a round time exceeds $2.5$ times the median separation time over the last 20 rounds. In particular, the last stopping criterion is used to avoid adding uninformative cuts, which, in turn, increase computational time and model size.  

\subsection{Root relaxation and facet audit}

We first evaluate the efficiency of the proposed DD-based cuts across different strategies and relative to cuts from the literature. Moreover, we use the facet dimension certificate from Section \ref{sec:dimension_facet_audit} to check the quality of the generated cuts and determine the portion of them that are facet-defining. 

Table~\ref{tab:root-main-final-candidate} reports the root-relaxation performance across all alternatives. Column ``Gap'' is the average root gap left after separation, while ``C\%'' reports the average percentage of the initial root gap closed by the
generated cuts~\cite{delPiaWalter2024simple}: for each instance, if $SR$
is the root objective value before adding our cuts, $DB$ is the root objective value
after separation, and $PRIMAL$ is the reference integer objective value, then
$ \mathrm{C\%} = 100\cdot \frac{DB-SR}{PRIMAL-SR}$.
Column ``Time'' is the mean root-separation time, while ``CM'' measures the average number of cuts generated, and ``C\% / 1000'' reports the mean percentage of gap closed divided by the mean number of cuts in thousands.

\begin{table}[t]
\centering
\caption{Root-relaxation performance on tested instances.}
\label{tab:root-main-final-candidate}
\begin{tabular}{llrrrrr}
\toprule
Dataset & Method & Gap & C\% & Time & CM & C\% / 1000 \\
\midrule
LABS & LT & 437.15 & 60.97 & 47.02 & 3412.4 & 17.87 \\
LABS & FS & 488.16 & 57.63 & 51.61 & 2801.4 & 20.57 \\
LABS & EX4 & 265.42 & 75.58 & 42.10 & 2092.7 & 36.12 \\
LABS & EX5 & 214.52 & 78.65 & 42.56 & 1892.8 & 41.55 \\
LABS & EX6 & 216.52 & 78.49 & 49.09 & 1615.8 & 48.57 \\
LABS & EX7 & 205.16 & 79.28 & 50.30 & 1670.5 & 47.46 \\
LABS & EX8 & 224.29 & 78.06 & 51.01 & 1532.7 & 50.93 \\
LABS & PT0 & 57.09 & 91.31 & 32.22 & 2202.5 & 41.46  \\
LABS & PT1 & 51.97 & 91.82 & 35.25 & 1677.7 & 54.73  \\
LABS & PT2 & 45.37 & 92.33 & 24.65 & 1175.5 & 78.55  \\
LABS & PT3 & 45.46 & 92.36 & 30.67 & 914.4 & 101.01  \\
LABS & PT4 & 43.28 & 92.40 & 28.92 & 686.5 & 134.59  \\
\addlinespace
VISION & LT & 5.86 & 95.82 & 21.25 & 1211.4 & 79.10 \\
VISION & FS & 9.17 & 93.44 & 13.13 & 1021.8 & 91.45 \\
VISION & EX4 & 9.68 & 93.15 & 17.67 & 862.9 & 107.95 \\
VISION & EX5 & 2.06 & 98.53 & 22.41 & 855.3 & 115.21 \\
VISION & EX6 & 0.55 & 99.59 & 16.43 & 548.7 & 181.52 \\
VISION & EX7 & 0.73 & 99.50 & 19.03 & 597.8 & 166.44 \\
VISION & EX8 & 1.69 & 98.83 & 21.00 & 607.3 & 162.74 \\
VISION & PT0 & 0.00 & 100.00 & 5.14 & 675.4 & 148.06 \\
VISION & PT1 & 0.04 & 99.98 & 5.14 & 662.9 & 150.82 \\
VISION & PT2 & 0.02 & 99.99 & 5.28 & 624.6 & 160.07 \\
VISION & PT3 & 0.19 & 99.87 & 5.81 & 662.6 & 150.73 \\
VISION & PT4 & 0.01 & 99.99 & 6.20 & 646.4 & 154.69 \\
\bottomrule
\end{tabular}
\end{table}

We first note that LT and FS have similar performance in both datasets, which is the intended validation outcome. This result shows that when the DD separator is given the same structural information as the literature separator, it recovers comparable root behavior. The smaller number of cuts in FS is likely due to cut collisions: if two literature structures share part of their support, DD separation may yield the same inequality over the shared region, so fewer distinct cuts are added.

In contrast, we observe significant improvement when using the DDs to represent more complex structures.  On LABS, the PT family dominates: PT4 gives the smallest final root gap, closing $92.40\%$ of the initial gap with only $686.5$ cuts on average, and it also has the largest Gap closed per 1000 cuts value. The EX family also improves substantially over LT: EX4 reduces the final average gap from $437.15$ to $265.42$, and EX7 yields the smallest EX final gap. On VISION, the result is even sharper: PT0 closes $100.00\%$ of the root gap across all tested instances, yielding perfect root relaxation on this benchmark set under our experimental configuration. The EX methods also nearly close the gap, with EX6 reaching $99.59\%$, but PT0 obtains the tightest relaxation with the fastest average separation time. 

We also observed that stronger supports change the nature of the generated cuts. In the EX family, increasing the expansion generally reduces the number of cuts while increasing the gap closed per 1000 cuts. This suggests that richer local DDs produce stronger individual inequalities, leaving fewer violated structures available after each round. The same pattern is even clearer in the PT family on LABS: as the $q$ parameter increases, the method produces fewer cuts but achieves higher efficiency. Thus, the larger local DDs are not merely producing more inequalities; they are producing more informative inequalities. Moreover, we observed that the computation time (i.e., column ``Time'') of all DD-based strategies is similar, reinforcing the practical benefits of DD-based cut generation strategies. 

\begin{table}[b]
\centering
\caption{Facet-audit summary on tested instances. }
\label{tab:facet-audit-root-summary}
\begin{tabular}{llrrrrr}
\toprule
Dataset & Method & Max \% & $I_K$ F\% & $I_K$ R\% & G F\% & GK \% \\
\midrule
LABS & LT & 27.4 & 0.0 & 53.3 & 0.0 & 85.7 \\
LABS & FS & 22.4 & 0.0 & 53.5 & 0.0 & 91.3 \\
LABS & EX4 & 38.3 & 31.6 & 90.9 & 31.6 & 97.5 \\
LABS & EX5 & 67.0 & 64.2 & 92.6 & 64.2 & 98.7 \\
LABS & EX6 & 99.8 & 97.8 & 92.9 & 97.8 & 99.6 \\
LABS & EX7 & 99.8 & 98.3 & 93.0 & 98.3 & 99.2 \\
LABS & EX8 & 99.8 & 98.3 & 93.0 & 98.3 & 99.7 \\
LABS & PT0 & 92.4 & 92.4 & 93.3 & 92.4 & 100.0 \\
LABS & PT1 & 95.0 & 94.9 & 95.0 & 94.9 & 100.0 \\
LABS & PT2 & 96.7 & 96.6 & 96.6 & 96.6 & 100.0 \\
LABS & PT3 & 96.6 & 96.5 & 97.2 & 96.5 & 100.0 \\
LABS & PT4 & 96.8 & 96.6 & 97.7 & 96.6 & 99.9 \\
\addlinespace
VISION & LT & 60.0 & 44.0 & 74.8 & 44.0 & 87.5 \\
VISION & FS & 57.7 & 51.8 & 79.6 & 51.8 & 95.0 \\
VISION & EX4 & 52.4 & 49.1 & 90.2 & 49.1 & 97.3 \\
VISION & EX5 & 78.6 & 52.6 & 79.0 & 52.7 & 78.9 \\
VISION & EX6 & 93.0 & 65.0 & 86.0 & 66.3 & 77.4 \\
VISION & EX7 & 100.0 & 59.5 & 88.1 & 67.2 & 67.5 \\
VISION & EX8 & 100.0 & 60.4 & 88.4 & 68.1 & 68.3 \\
VISION & PT0 & 86.7 & 52.5 & 88.9 & 59.1 & 72.5 \\
VISION & PT1 & 93.0 & 53.1 & 88.5 & 60.6 & 67.9 \\
VISION & PT2 & 95.5 & 57.1 & 88.8 & 64.5 & 69.3 \\
VISION & PT3 & 94.2 & 55.8 & 88.9 & 62.7 & 69.2 \\
VISION & PT4 & 96.2 & 58.8 & 89.3 & 65.8 & 70.0 \\
\bottomrule
\end{tabular}
\end{table}

To further analyze the quality of the generated cuts, we apply our DD-based dimension certificate procedure (i.e., Algorithms \ref{alg:dd_dimension_audit} and~\ref{alg:ik_certificate_audit}) to all generated cuts and report the results in Table~\ref{tab:facet-audit-root-summary}. Column ``Max \%'' is the percentage of audited cuts whose support is maximal in the selected local hypergraph. Column ``$I_K$ F\%'' is the percentage certified as facet-defining for the local $I_K$ relaxation by Algorithm~\ref{alg:ik_certificate_audit}. Column ``$I_K$ R\%'' measures the average ratio between the certified face dimension and the target local dimension, and ``G F\%'' is a conservative lower bound on the percentage of cuts certified as globally facet-defining. Lastly, ``GK \%'' reports the percentage of audited cuts for which the global status could be resolved within the audit budget. If the $I_K$ test does not certify facetness, a global audit is attempted only when the corresponding global DD can be built and audited within a 60\,s limit; otherwise, the global status is left unresolved. The diagnostic audit time is excluded from the reported algorithmic wall-clock times.

Table~\ref{tab:facet-audit-root-summary}  supports the results reported in Table~\ref{tab:root-main-final-candidate}. On LABS, LT and FS generate many cuts, but essentially none are certified as $I_K$ facets. By contrast, EX6 and larger expansions produce cuts that are almost always maximal and overwhelmingly $I_K$-facet-defining, while the PT variants achieve similarly strong certification rates. On VISION, the audit gives a more nuanced but still positive picture. The strongest expanded supports, especially EX7 and EX8, reach $100.0\%$ maximality and have the largest global-certification lower bounds among the EX methods, but their global status is often unresolved within the conservative 60\,s audit budget. The PT methods have lower local facet rates than on LABS, but still certify a large fraction of cuts and, as Table~\ref{tab:root-main-final-candidate} shows, PT0 closes the root gap completely on the tested VISION instances. Taken together, the root and audit results support the main claim of the paper: DD separation is most useful not as another named family of inequalities, but as a black-box mechanism that can take literature-derived structures, expand the support seen by the separator, and produce stronger inequalities than the original structures alone.

\subsection{B\&B performance}

We next evaluate whether the stronger root relaxations translate into better B\&B performance. To ensure a clean comparison, these experiments consider cuts added at the root node by a cutting-plane procedure; that is, root separation is performed on an explicit LP relaxation before tree search begins. Thus, all found cuts are added to this LP model, and we then run Gurobi's B\&B procedure over its MILP version. 

\begin{table}[t]
\centering
\caption{Branch-and-bound performance on tested instances.}
\label{tab:bb-main-final-candidate}
\begin{tabular}{llrrrrr}
\toprule
Dataset & Method & Solved & Time & Nodes & Gap \% & C\% \\
\midrule
LABS & LT & 11 & 711.49 & 1717587 & 106.38 & 60.97 \\
LABS & EX4 & 10 & 466.93 & 1412006 & 60.54 & 75.58 \\
LABS & EX5 & 10 & 548.47 & 1768355 & 58.04 & 78.65 \\
LABS & EX6 & 10 & 575.08 & 1779272 & 49.98 & 78.49 \\
LABS & EX7 & 11 & 687.33 & 1667117 & 54.06 & 79.28 \\
LABS & EX8 & 10 & 520.01 & 1771057 & 52.00 & 78.06 \\
LABS & PT0 & 11 & 170.17 & 150556 & 37.17 & 91.31 \\
LABS & PT1 & 11 & 492.98 & 72437 & 51.48 & 91.82 \\
LABS & PT2 & 11 & 689.01 & 47030 & 69.64 & 92.33 \\
LABS & PT3 & 10 & 496.79 & 79766 & 47.30 & 92.36 \\
LABS & PT4 & 10 & 527.26 & 57131 & 44.09 & 92.40 \\
\addlinespace
VISION & LT & 40 & 71.79 & 96337 & 2.41 & 95.82 \\
VISION & EX4 & 36 & 218.73 & 573610 & 4.61 & 93.15 \\
VISION & EX5 & 45 & 31.16 & 1342 & -- & 98.53 \\
VISION & EX6 & 45 & 17.36 & 2 & -- & 99.59 \\
VISION & EX7 & 45 & 19.99 & 4 & -- & 99.50 \\
VISION & EX8 & 45 & 22.05 & 6 & -- & 98.83 \\
VISION & PT0 & 45 & 5.76 & 1 & -- & 100.00 \\
VISION & PT1 & 45 & 5.89 & 2 & -- & 99.98 \\
VISION & PT2 & 45 & 6.03 & 1 & -- & 99.99 \\
VISION & PT3 & 45 & 7.08 & 1 & -- & 99.87 \\
VISION & PT4 & 45 & 7.35 & 1 & -- & 99.99 \\
\bottomrule
\end{tabular}
\end{table}

Table~\ref{tab:bb-main-final-candidate} reports B\&B performance of all procedures except FS, which was evaluated only at the root node as a validation benchmark. In the table, column ``Solved'' gives the number of instances solved within the 3600\,s limit. Column ``Time'' corresponds to the average solve time computed only over solved runs, ``Nodes'' reports the mean number of B\&B nodes, and ``Gap \%'' is the mean final gap percent computed only over instances that hit the time limit. The last column repeats the root gap closed (i.e.,  ``C\%''), allowing us to compare root strength against downstream tree-search behavior.

On LABS instances, PT0 is the strongest practical B\&B setting in terms of instances solved, time, and gap. This is notable because PT0 is not the strongest PT configuration at the root. The result suggests that the cuts that close the most root gap are not necessarily the cuts that produce the best B\&B behavior. In particular, this result can be explained by the fact that PT0 yields significantly more cuts than the other PT strategies (see Table \ref{tab:root-main-final-candidate}), which, in turn, can help the B\&B algorithm prove optimality. Nonetheless, we note that all PT strategies perform quite well when compared to the EX strategies and the LT baseline.

On VISION, the result obtained at the root node essentially carries over to B\&B. PT strategies solve all 45 instances at the root node and yield the fastest solve times. This is consistent with the root table, where PT0 closes $100.00\%$ of the root gap on VISION, and the other alternatives more than $99.8\%$. The EX methods also perform strongly once the expansion is large enough: EX5--EX8 solve all instances, and EX6 is the fastest expanded-support method. Overall, the B\&B results reinforce the main computational message: DD support expansion improves performance, but the best strategy is not necessarily the one that maximizes root closure.

\begin{figure}[t]
    \centering
    \includegraphics[width=0.96\textwidth]{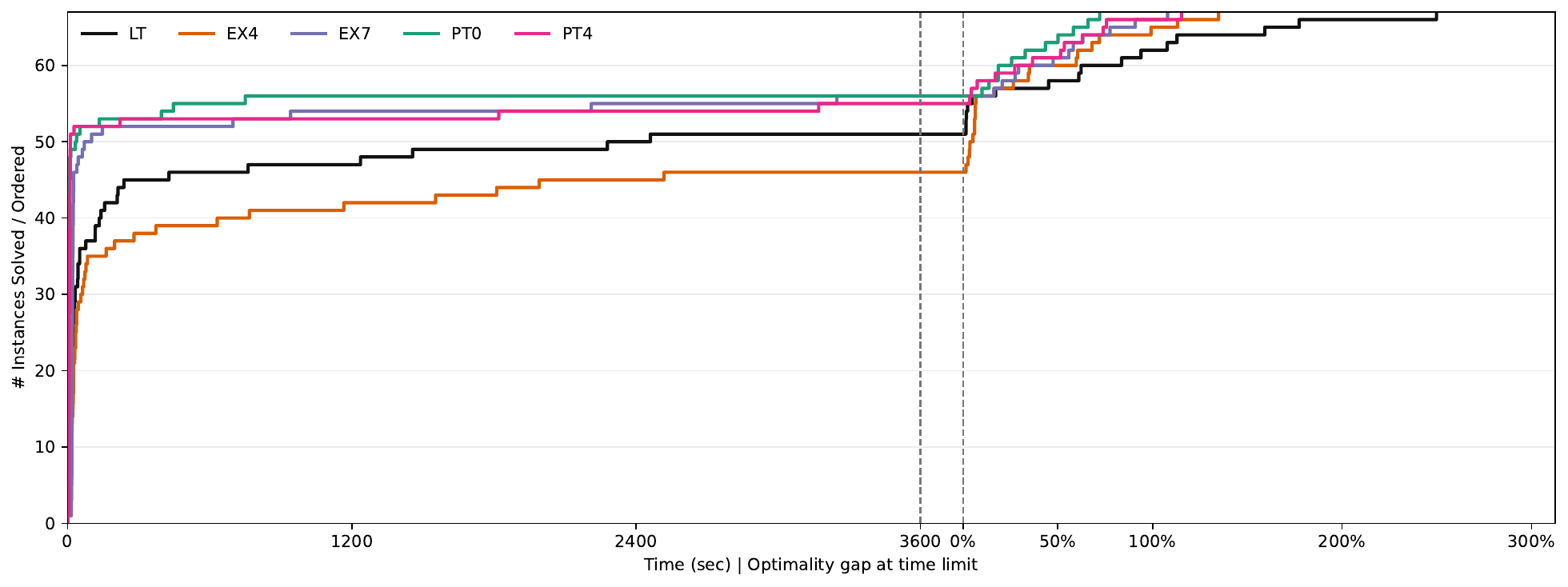}
    \caption{B\&B profile for representative methods.}
    \label{fig:bb-profile-final}
\end{figure}

Figure~\ref{fig:bb-profile-final} summarizes the same behavior across the two datasets. The plot depicts only two strategies per family (i.e., EX4, EX7, PT0, PT4) to avoid overcrowding. The left part of the plot shows the cumulative number of solved instances over time, while the right part is a cumulative gap plot for the instances that reach the time limit. Once again, we see that the PT strategies and the larger EX excel in both runtime and optimality gap compared to our baseline LT. Also, it is interesting to see that EX4 has poor performance, which is mostly explained by the fact that the underlying DDs do not represent structures as rich as those of the other alternatives, thus, providing weaker cuts as shown in Tables \ref{tab:root-main-final-candidate} and \ref{tab:facet-audit-root-summary}.

\subsection{Discussion}

These results show that the main gains from our procedures stem from strengthening the support prior to separation. Expanded supports improve substantially over the fixed-support baseline as $M$ increases, and the partition supports are the strongest family overall. On LABS, the PT family yields the best root gaps and the best B\&B behavior among the tested DD-based strategies; in particular, PT0 achieves the fastest solved-instance time and the smallest remaining gap among the unsolved instances, while deeper partition levels close slightly more of the root gap but do not translate into better B\&B performance. On VISION, the same pattern is even sharper: PT0 closes the full root gap on all tested instances and gives the fastest solve times. On VISION, EX6 is the strongest configuration among the EX variants, but the partition supports dominate the overall operational comparison.

In a nutshell, DDs provide a mechanism for turning a chosen support into a separation oracle. When that support is the classical literature structure, the method validates itself by producing cuts of comparable strength. When the support is expanded or repartitioned, the same mechanism can generate stronger cuts than the original structure alone could. This supports the interpretation of DDs as a black-box strengthening device: rather than deriving a new family of inequalities by hand for every enlarged structure, we specify the structure and let the DD-based target-cut model separate over it.
\section{Conclusions}
\label{sec:discussion}

This paper studied binary polynomial optimization (BPO) through the lens of its multilinear polytope and developed a decision-diagram (DD) framework for generating cutting planes over local structures. We presented a compact DD encoding of the multilinear set that considers only vertex variables and implicitly recovers the hyperedge variables along states and arcs. This encoding yields an extended network-flow formulation of $MP_G$ and supports a target-cut separation model adapted to the implicitly assigned hyperedge variables. We analyzed the size of this construction via a valid-antichain characterization of its width, which allowed us to identify hypergraph structures---including flowers and cycles---that admit bounded-width DDs, as well as a new family of hypergraphs for which BPO is solvable in polynomial time. Building on this machinery, we proposed three strategies for choosing the local support fed to the separator: a fixed-support baseline that mirrors known inequalities in the literature, an expanded-support rule that grows a seed by overlap, and a partition-support scheme based on section hypergraphs. Finally, we introduced a novel DD-based dimension and facet-audit procedure to certify the polyhedral quality of the generated cuts.

Computationally, on the LABS and VISION datasets, the fixed-support DD separator reproduces the behavior of the corresponding literature cuts, confirming that the target-cut model recovers comparable strength when given the same local information. The main gains, however, come from strengthening the support before separation: expanded and, especially, partition supports close substantially more of the root gap with far fewer cuts---on LABS the partition family closes up to $92.4\%$ of the root gap, and on VISION it closes the gap almost completely on every tested instance---and these stronger relaxations translate into a faster branch-and-bound, with the partition setting solving all VISION instances and giving the best practical performance on LABS. The facet audit corroborates this picture: strengthened supports yield cuts that are far more often maximal and facet-defining than those from the literature. Taken together, the results support the central message of the paper: a DD is best viewed not as another hand-derived family of inequalities, but as a black-box strengthening device that turns any chosen support into a facet-producing separation oracle.

Several limitations remain. The current implementation generates cuts only at the root, uses offline support choices, and relies on heuristic selection of structure. Future work should integrate the separator into a full branch-and-cut routine, combine DD cuts with the solver's native cuts, and develop adaptive criteria for selecting supports during the solve. Another promising direction is to use the relatively cheap $I_K$-facet audit as a cut-selection tool, prioritizing cuts that are not only violated but also locally facet-defining. Finally, the randomized nature of the PT family suggests studying multi-seed or adaptive partitioning schemes that preserve PT0's strong performance while improving robustness.

More broadly, the same approach could be applied to other binary optimization problems where strong inequalities arise from small but meaningful supports and, thus, enhance modern solvers such as Gurobi. Moreover, our procedure can be seen as a research tool in its own right. Specifically, it can help researchers identify interesting problem structures (i.e., for BPO or other problems), obtain strong cuts from them, and then infer properties and develop algorithms to generate those same cuts without needing to create a DD.

\begin{acknowledgements}
This research was partially supported by the National Center for Artificial Intelligence CENIA FB210017 (Basal ANID), and the \textit{Agencia Nacional de Investigación y Desarrollo} of Chile under grant ANID-FONDECYT-1261833.
\end{acknowledgements}

\section*{Conflict of interest}
The authors declare that they have no conflict of interest.

\bibliographystyle{spmpsci}
\bibliography{references}

\end{document}